%% file: coroas-arxiv.tex
\documentclass[final,12pt]{elsarticle}

\makeatletter
\def\ps@pprintTitle{%
 \let\@oddhead\@empty
 \let\@evenhead\@empty
 \def\@oddfoot{}%
 \let\@evenfoot\@oddfoot}
\makeatother

\usepackage{algorithmicx,algorithm, algpseudocode, refcount}
\usepackage{comment}
\usepackage{tikz}
\usepackage{rotating}
\usetikzlibrary{patterns}
\usetikzlibrary{arrows.meta}
\usetikzlibrary{positioning}
\usetikzlibrary{calc}
\usetikzlibrary{fadings}
\usetikzlibrary{patterns}
\usetikzlibrary{shadows.blur}
\usetikzlibrary{shapes}
\usepackage{mathdots}
\usepackage{yhmath}
\usepackage{cancel}
\usepackage{color}
\usepackage{siunitx}
\usepackage{array}
\usepackage{multirow}
\usepackage{amssymb}
\usepackage{tabularx}
\usepackage{amsthm}
\usepackage{tikzit}

\input{graphs.tikzstyles}
\input{graphs-tikz}

\usepackage{paralist}

\usepackage{amsfonts}
\usepackage{epsfig}
\usepackage{multicol}
\usepackage{amsmath,amssymb,amssymb,latexsym}
\usepackage{enumerate}
\usepackage{graphicx}
\usepackage{mathtools}

\DeclareMathOperator{\thin}{thin}
\DeclareMathOperator{\pthin}{pthin}
\DeclareMathOperator{\indthin}{thin_{ind}}

\DeclareMathOperator{\indpthin}{pthin_{ind}}
\DeclareMathOperator{\compthin}{thin_{cmp}}
\DeclareMathOperator{\comppthin}{pthin_{cmp}}

\newcommand{\N}{\mathbb{N}}

\DeclareMathOperator{\fpp}{prec-pthin}
\DeclareMathOperator{\fp}{prec-thin}
\DeclareMathOperator{\indfpp}{prec-pthin_{ind}}
\DeclareMathOperator{\indfp}{prec-thin_{ind}}
\DeclareMathOperator{\compfpp}{prec-pthin_{cmp}}
\DeclareMathOperator{\compfp}{prec-thin_{cmp}}
\DeclareMathOperator{\mirror}{mirror}
\DeclareMathOperator{\side}{side}
\DeclareMathOperator{\Little}{Little}
\DeclareMathOperator{\first}{first}
\DeclareMathOperator{\last}{last}

\newcommand{\pptFull}[1]{precedence proper $#1$-thin}

\newcommand{\ptFull}[1]{precedence $#1$-thin}

\newcommand{\set}[1][ ]{\{ #1 \}}

\newtheorem{theorem}{Theorem}
\newtheorem{proposition}[theorem]{Proposition}
\newtheorem{condition}{Condition}
\newtheorem{lemma}[theorem]{Lemma}
\newtheorem{corollary}[theorem]{Corollary}
\newtheorem{claim}[theorem]{Claim}

\theoremstyle{definition}

\theoremstyle{remark}

\begin{document}

\begin{frontmatter}
\title{Thinness and its variations on some graph families and coloring graphs of bounded thinness}

\author[UBA,ICC]{Flavia Bonomo-Braberman}
\ead{fbonomo@dc.uba.ar}

\author[UBA]{Eric~Brandwein}
\ead{ebrandwein@dc.uba.ar}

\author[UERJ]{Fabiano~S.~Oliveira}
\ead{fabiano.oliveira@ime.uerj.br}

\author[UFRJ]{Moysés~S.~Sampaio~Jr.}
\ead{moysessj@cos.ufrj.br}

\author[UBA]{Agustin~Sansone}
\ead{asansone@dc.uba.ar}

\author[UERJ,UFRJ]{Jayme~L.~Szwarcfiter}
\ead{jayme@nce.ufrj.br}

\address[UBA]{Universidad de Buenos Aires. Facultad de Ciencias Exactas y Naturales. Departamento de Computaci\'on. Buenos Aires,
Argentina.}
\address[ICC]{CONICET-Universidad de Buenos Aires. Instituto de
Investigaci\'on en Ciencias de la Computaci\'on (ICC). Buenos
Aires, Argentina.}
\address[UFRJ]{Universidade Federal do Rio de Janeiro, Brazil.}
\address[UERJ]{Universidade do Estado do Rio de Janeiro, Brazil.}

%
\begin{abstract} 
Interval graphs and proper interval graphs are well known graph classes, for which several generalizations have been proposed in the literature. In this work, we study the (proper) thinness, and several variations, for the classes of cographs, crowns graphs and grid graphs.

We provide the exact values for several variants of thinness (proper, independent, complete, precedence, and combinations of them) for the crown graphs $CR_n$. For cographs, we prove that the precedence thinness can be determined in polynomial time. 
We also improve known bounds for the thinness of $n \times n$ grids $GR_n$ and $m \times n$ grids $GR_{m,n}$, proving that $\left \lceil \frac{n-1}{3} \right \rceil \leq \thin(GR_n) \leq  \left \lceil \frac{n+1}{2} \right \rceil$. Regarding the precedence thinness, we prove that $\fp(GR_{n,2}) = \left \lceil \frac{n+1}{2} \right \rceil$ and that $\left \lceil \frac{n-1}{3} \right \rceil \left \lceil\frac{n-1}{2} \right \rceil + 1 \leq \fp(GR_n) \leq \left \lceil\frac{n-1}{2} \right \rceil^2+1$. As applications, we show that the $k$-coloring problem is NP-complete for precedence $2$-thin graphs and for proper $2$-thin graphs, when $k$ is part of the input. On the positive side, it is polynomially solvable for precedence proper 2-thin graphs, given the order and partition. 
\end{abstract}
%
%
\end{frontmatter}
%
\section{Introduction}

This work studies the classes of $k$-thin and proper $k$-thin graphs, which generalize interval and proper interval graphs, respectively. A graph $G$ is an \emph{interval graph} if there is a mapping of $V(G)$ on intervals of the real line, each  $v \in V(G)$ corresponding to the interval $I_v$, such that $(u,v) \in E(G)$ if and only if $I_u \cap I_v \neq \emptyset$ for all distinct $u,v \in V(G)$. Moreover, $G$ is called a \emph{proper interval graph} if there is such a mapping of vertices to intervals such that $I_u \not \subseteq I_v$ for all distinct $u,v \in V(G)$. There are several characterizations and recognition algorithms for interval and proper interval graphs~\cite{Ola-interval, Rob-box, B-L-PQtrees}. A well known characterization for interval graphs is the following:

\begin{theorem}[\hspace{1sp}\cite{Ola-interval}] \label{teo:IG_Carac}
    A graph $G$ is an interval graph if and only if there is an ordering $\sigma$ of $V(G)$ such that, for any triple $(p,q,r)$ of vertices of $V(G)$ ordered according to $\sigma$, if $(p,r) \in E(G)$, then $(q,r)\in E(G)$.
\end{theorem}

The ordering $\sigma$ of Theorem~\ref{teo:IG_Carac} is said to be a \emph{canonical} ordering or interval ordering. An analogous characterization for proper interval graphs is described next.

\begin{theorem}[\hspace{1sp}\cite{Rob-box}]\label{teo:PIG_Carac}
    A graph $G$ is a proper interval graph if and only if there is an ordering $\sigma$ of $V(G)$ for which $\sigma$ and its reversal are both canonical.
\end{theorem}

The ordering $\sigma$ of Theorem~\ref{teo:PIG_Carac} is said to be a \emph{proper canonical} ordering or proper interval ordering.

A \emph{$k$-thin graph} $G$ is a graph for which there is a $k$-partition $\mathcal{V} = (V_1,V_2 \ldots, V_k)$ of $V(G)$, and an ordering $\sigma$ of $V(G)$ such that, for any triple $(p,q,r)$ of vertices of $V(G)$ ordered according to $\sigma$, if $p, q \in V_i$ for some $1 \leq i \leq k$ and $(p,r) \in E(G)$, then $(q,r)\in E(G)$. Such an ordering $\sigma$  is said to be \emph{consistent} with $\mathcal{V}$. A graph $G$ is called a \emph{proper $k$-thin graph} if $V(G)$ admits a $k$-partition $\mathcal{V}$ of $V(G)$ and an ordering $\sigma$ such that both $\sigma$ and its reversal are consistent with $\mathcal{V}$. An ordering of this type is said to be \emph{strongly consistent} with $\mathcal{V}$. 

Although not all graphs $G$ are (proper) interval graphs, all of them are (proper) $k$-thin graphs for some $k \geq 1$. This is so because if $k = |V(G)|$, any ordering of $V(G)$ is (strongly) consistent with the unique possible partition in which each part consists of a single vertex. Moreover, since (proper) interval graphs are precisely the (proper) $1$-thin graphs, the parameter $k$ of (proper) $k$-thin graphs measures ``how far'' a graph is from being a (proper) interval graph. 

Under such classes, some NP-complete problems can be solved in polynomial time, as proved in ~\cite{M-O-R-C-thinness, B-D-thinness, B-M-O-thin-tcs}.  For instance, generalized versions of $k$-coloring are polynomial-time solvable for graphs of bounded thinness, when the number of colors $k$ is also a  constant~\cite{B-D-thinness,B-M-O-thin-tcs}. 
Despite that, the complexity of $k$-coloring on graphs with bounded thinness at least two is open when $k \geq 3$ is part of the input (it is polynomial-time solvable for 1-thin graphs, i.e., interval graphs~\cite{Go-perf}). 

Variations of the concept of (proper) thinness have been studied in the literature, 
by constraining either the vertices that can share a part in the partition, or how vertices can be arranged in the ordering. The  classes of \emph{\ptFull{k}} and \emph{\pptFull{k} graphs} were defined in~\cite{BOSS-thin-prec-dam-lagos}, as (proper) $k$-thin graphs such that the vertices belonging to a same part must be consecutive in the (strongly) consistent ordering.
It still holds that precedence (proper) $1$-thin  graphs are equivalent to (proper) interval graphs. 
In~\cite{BOSS-thin-prec-dam-lagos}, among other results, it was presented a characterization of such classes based on threshold graphs. In~\cite{BGOSS-thin-oper}, there have been defined the classes of  (proper) $k$-independent-thin and (proper) $k$-complete-thin graphs, for which each part of the partition must be an independent and a complete set, respectively. The authors proved some bounds on the thinness for such variants for some graphs operations. In~\cite{B-B-lagos21-dam}, it is proved that $2$-independent-thin  graphs are equivalent to interval bigraphs, and that proper independent 2-thin graphs are equivalent to bipartite permutation graphs, so (proper) $k$-independent-thin graphs can be seen as generalizations of those classes.

The \emph{crown graph} $CR_n$ is the graph obtained from a complete bipartite graph $K_{n,n}$ by removing a perfect matching. It was proven in \cite{BGOSS-thin-oper} that $\thin(CR_n) \geq \frac{n}{2}$. 
The \emph{grid graph} $GR_{n,m}$ is defined as the Cartesian product of two path graphs $P_n$ and $P_m$. We denote $GR_{n,n}$ simply by $GR_n$. It was proved in~\cite{M-O-R-C-thinness} that $\frac{n}{4} \leq \thin(GR_n) \leq n+1$. In~\cite{Agus-Eric-tesis}, the upper bound for $\thin(GR_n)$ was improved to $\left \lceil \frac{2n}{3} \right \rceil$.

In this work, we provide a characterization of consistent solutions for the crown graphs $CR_n$, and solve all known variants of the thinness parameter for this class.  With respect to the grid graphs, we prove that the thinness of $GR_n$ is at least $\left \lceil \frac{n-1}{3} \right \rceil$ and at most $\left \lceil \frac{n+1}{2} \right \rceil$, which are tighter lower and upper bounds than the ones found in the literature. Also, we show that the precedence thinness of $GR_{2,n}$ is exactly $\left \lceil \frac{n+1}{2} \right \rceil$, whereas for $GR_n$ is between $\left \lceil \frac{n-1}{3} \right \rceil \left \lceil\frac{n-1}{2}\right \rceil + 1$ and $\left \lceil\frac{n-1}{2}\right \rceil^2+1$. Finally, we prove that the $k$-coloring problem is NP-complete for both \ptFull{2} and proper $2$-thin graphs when $k$ is part of the input, but polynomially solvable for \pptFull{2} graphs.

This paper is structured as follows. Section~\ref{sec:pre} presents the remaining definitions used throughout the paper. In Section~\ref{sec:crown}, we present a characterization of consistent solutions for crown graphs and solve all known variants of the thinness parameter for this class, summarized in Table~\ref{table1}. Section~\ref{sec:cographs} provides formulas for precedence thinness of the disjoint union and join of two graphs in terms of their precedence thinness, thus proving that it is possible to efficiently compute the precedence thinness of cographs. In Section~\ref{sec:grid}, we prove lower and upper bounds for the thinness and precedence thinness of grid graphs, improving the previously known ones. In Section~\ref{sec:color}, we prove that the $k$-coloring problem is NP-complete for both \ptFull{2} and proper $2$-thin graphs when $k$ is part of the input, but polynomially solvable for \pptFull{2} graphs. Also, we show that \pptFull{2} graphs are perfectly orderable. Finally, some concluding remarks are presented in Section~\ref{sec:conclusion}.

\section{Preliminaries} \label{sec:pre}

Let $G$ be a graph. We denote by $V(G)$ its vertex set and by $E(G)$ its edge set.
The \emph{size} (\emph{cardinality}) of a set $S$ is denoted by $|S|$. Let $u,v \in V(G)$, $u$ and $v$ are \emph{adjacent} if $(u,v) \in E(G)$.

Two graphs $G$ and $H$ are said to be \emph{isomorphic}, denoted by $G \cong H$, if there is a bijection $\theta : V(G) \rightarrow V(H)$ such that $(u,v) \in E(G)$ if, and only if, $(\theta(u),\theta(v)) \in E(H)$.

The \emph{subgraph} of $G$ \emph{induced} by the subset of vertices $V'\subseteq V(G)$, denoted by $G[V']$, is the graph $G[V']= (V',E')$, where  $E' =  \{ (u,v) \in E(G) \mid u, v \in V' \}$. An \emph{induced subgraph} of $G$ is a subgraph of $G$ induced by some subset of $V(G)$. 

Let $u \in V(G)$. The \emph{(open) neighborhood} of $u$, denoted by $N_G(u)$, is defined as $N_G(u) = \set[v \in V(G) \mid (u, v) \in E(G)]$. The \emph{closed neighborhood} of $u$ is denoted by $N_G[u] = \set[u] \cup N_G(u)$. Let $U \subseteq V(G)$. Define $N_G(U) = \bigcup_{u \in U} N_G(u)$ and $N_G[U] = \bigcup_{u \in U} N_G[u]$. When $G$ is clear in the context, it may be omitted from the notation. 

A \emph{clique} or \emph{complete set} (resp.\ \emph{stable set} or \emph{independent
set}) is a set of pairwise adjacent (resp.\ nonadjacent) vertices.  The clique graph $K_n$ is the graph such that $V(K_n)$ is a clique.

Let $G=(V,E)$, the \emph{complement} of $G$ is defined as $\overline{G}=(V,E')$ such that $E'=\{(u, v) \mid u,v \in V \text{ and } (u,v) \not\in E \}$.

The \emph{union} of $G_1$ and $G_2$ is
the graph $G_1 \cup G_2 = (V_1 \cup V_2, E_1 \cup E_2)$, and the
\emph{join} of $G_1$ and $G_2$ is the graph $G_1 \vee G_2 = (V_1
\cup V_2, E_1 \cup E_2 \cup \{(v,v') : v \in V_1, v' \in V_2\})$ (i.e.,
$\overline{G_1\vee G_2} = \overline{G_1} \cup \overline{G_2}$).
(The join is sometimes also noted by $G_1 \otimes G_2$, but we
follow the notation in~\cite{B-D-thinness}). 

The class of \emph{cographs} can be defined as the graphs that can
be obtained from trivial graphs by the union and join
operations~\cite{CorneilLerchsStewart81}.

The \emph{thinness} (resp. \emph{proper thinness}) of $G$, or $\thin(G)$ (resp. $\pthin(G)$), is defined as the minimum value $k$ for which $G$ is a $k$-thin graph (resp. proper $k$-thin graph). Similarly, the other types of thinness have their related parameters, listed next.  With respect to each restriction, we have

\begin{itemize}
    \item the \emph{precedence thinness} of $G$, or $\fp(G)$

    \item the \emph{independent thinness} of $G$, or $\indthin(G)$

    \item the \emph{complete thinness} of $G$, or $\compthin(G)$

     \item the \emph{precedence proper thinness} of $G$, or $\fpp(G)$

    \item the \emph{independent proper thinness} of $G$, or $\indpthin(G)$

    \item the \emph{complete proper thinness} of $G$, or $\comppthin(G)$
    
\end{itemize}

Also, as some of those restrictions can be combined, this leads to the following parameters: 
\begin{itemize}
    \item the \emph{precedence independent thinness} of $G$, or $\indfp(G)$

    \item the \emph{precedence complete thinness} of $G$, or $\compfp(G)$

    \item the \emph{precedence independent proper thinness} of $G$, or $\indfpp(G)$

    \item the \emph{precedence complete proper thinness} of $G$, or $\compfpp(G)$
    
\end{itemize}

A graph $G$ is \emph{bipartite} if its
vertex set can be partitioned into two independent sets, and this is denoted by $G=(V_1 \cup V_2,E)$, where $V_1$ and $V_2$ are independent. The \emph{complete bipartite} graph $K_{p,q}$ is a bipartite graph such that $|V_1|=p$, $|V_2|=q$, and for all $v_1 \in V_1$ and $v_2 \in V_2$, $(v_1,v_2) \in E$.

For a graph $G$ and $A, B \subseteq V(G)$ such that
$A \cap B = \emptyset$, the \emph{bipartite graph induced} by the two subsets, denoted by
$G[A,B]$, is the bipartite graph $G[A,B]=(A \cup B,E')$, where $E' \subseteq E(G)$ are the edges with one endpoint
in $A$ and one endpoint in $B$. 

A \emph{matching} $\mathcal{M}$ of a graph $G$ is defined as a set of pairwise non-adjacent edges of $G$. If  all vertices in $V(G)$ are endpoints of the edges in $\mathcal{M}$, then $\mathcal{M}$ is said to be a \emph{perfect matching}.

Let $\mathcal{M}$ be a perfect matching of $K_{n,n}$. The \emph{crown graph} $CR_n$ (also known as \emph{Hiraguchi graph}) is the graph on $2n$ vertices such that $V(CR_n) = V(K_{n,n})$ and $E(CR_n) = E(K_{n,n})) \setminus \mathcal{M}$. 

For positive integers $n$, $m$, the $(n \times m)$\emph{-grid}, noted $GR_{n,m}$, is the graph whose vertex set is $\{ (i,j): 1 \leq i \leq n, 1 \leq j \leq m \}$
and whose edge set is $\{ ((i,j),(k,l)) : |i-k| + |j-l| = 1, \text{ where } 1 \leq i,k \leq n, 1 \leq j,l \leq m \}$. We denote simply by $GR_n$ the $(n \times n)$-grid $GR_{n,n}$.





Let $\sigma$ be a sequence $(e_1,e_2,\ldots,e_k)$. We denote by $\overline{\sigma} = (e_k,e_{k-1},\ldots,e_1)$ the \emph{reversal} of $\sigma$, and by $\sigma_1 \oplus \sigma_2$ (or $\sigma_1\sigma_2$ ) the concatenation of the sequences $\sigma_1$ and $\sigma_2$. If $a$ precedes $b$ in $\sigma$, then we will denote $a <_{\sigma} b$. Given an ordering $<$ (that is, a linear order)  of a set of elements $X$, $\sigma_<$ denotes the sequence of vertices ordered according to $<$. 

For $S \subseteq X$, let $\first(S)$ and $\last(S)$ be the smallest and the greatest elements of $S$ according to $<$, respectively.
We say $S$ is \emph{consecutive in $X$ according to
$<$} if there is no $z$ in $X \setminus S$ such that $\first(S) < z
< \last(S)$. Notice that for each $u \in S$, there are at
most two elements $x \in X$ such that $\{u, x\}$ is consecutive in
$S$ according to $<$. Namely, if the elements of $X$ are ordered $s_1 < \dots < s_r$ and $u =
s_i$, such two elements are $s_{i-1}$ when $i > 1$ and $s_{i+1}$ when
$i < r$.

\begin{lemma}[Alternative characterization for strong consistency \cite{B-B-M-P-convex-jcss}]\label{prop:neigh:propthin} 
Given a graph $G$, an ordering $<$ of $V(G)$ and a partition $\mathcal{V}$ of $V(G)$ are strongly consistent if and only if for every $v \in V(G)$ and every $V \in \mathcal{V}$, the set $N[v] \cap (V \cup \{v\})$ is consecutive in $V \cup \{v\}$ according to $<$. In particular, $N[v] \cap V$ is consecutive in $V$ according to $<$.
\end{lemma}

Given a graph $G$, a partition $\mathcal{V}$ of $V(G)$ and an ordering $\sigma$ of $V(G)$, we say that the pair $(\mathcal{V}, \sigma)$ is a \emph{solution} which is \emph{(strongly) consistent} if $\sigma$ is a (strongly) consistent ordering with respect to $\mathcal{V}$ and $G$. 

Given an order $<$ and a partition of the vertices, we say that a triple $(r, s, t)$ \emph{breaks the consistency} meaning that $r < s < t$, $r$ and $s$ belong to the same class, and there is an edge between $r$ and $t$ but there is no edge between $s$ and $t$.

Let $G$ be a graph, $H$ be an induced subgraph of $G$, and let $f$  be one of the parameters under study in this paper (i.e., those appearing in Table~\ref{table1}).  
An observation that we will use often in the proofs of lower bounds is that $f(H) \leq f(G)$. This holds because the order and partition of $V(G)$ involved in the definition of $f$, when restricted to $V(H)$, preserve all the desired properties in $H$, i.e., consistency, strong consistency, precedence, and independence or completeness of the classes.

A \emph{coloring} of a graph $G$ is a labeling of its vertices such that adjacent vertices have distinct labels, regarded as colors. A
\emph{$k$-coloring} is a coloring that maps the vertex set into a
set of size $k$. Along this paper, we will regard a $k$-coloring
of a graph $G$ as a function $f:V(G)\to\N$ such that $f(v) \leq k$
for every $v\in V(G)$, and $f(u) \neq f(v)$ for adjacent vertices $u$ and $v$. A graph is
\emph{$k$-colorable} if it admits a $k$-coloring. 
The \emph{$k$-coloring problem}
takes as input a graph $G$ and a natural number $k$, and consists
of deciding whether $G$ is $k$-colorable. 

It is known that generalized versions of $k$-coloring are polynomial-time solvable for graphs of bounded thinness, when $k$ is a constant~\cite{B-D-thinness,B-M-O-thin-tcs,C-T-V-Z--H-top-thin}. But the complexity of $k$-coloring on graphs with bounded thinness at least two remained open when $k$ is part of the input (it is polynomial-time solvable for 1-thin graphs, i.e., interval graphs~\cite{Go-perf}). 

\section{The thinness of crown graphs} \label{sec:crown}
In this section, we prove the exact value of the thinness and its variations for $CR_n$, along with consistent orderings and partitions in each case. The main results are summarized in Table~\ref{table1}. 

\begin{table}
\begin{center}
\begin{tabular}{|c|c|c|c|c|c|}
  \hline
   & $CR_1$ & $CR_2$ & $CR_3$ & $CR_n$, $n\geq 4$ even & $CR_n$, $n\geq 4$ odd \\
  \hline
  $\thin$ & $1$ & $1$ & $2$ & $n-1$ & $n-1$ \\
  $\pthin$ & $1$ & $1$ & $2$ & $n-1$ & $n$ \\
  $\indthin$ & $1$ & $2$ & $3$ & $n$ & $n$ \\
  $\indpthin$ & $1$ & $2$ & $3$ & $n$ & $n$ \\
  $\compthin$ & $2$ & $2$ & $3$ & $2n-4$ & $2n-4$ \\
  $\comppthin$ & $2$ & $2$ & $3$ & $2n-4$ & $2n-4$ \\
  $\fp$ & $1$ & $1$ & $2$ & $n-1$ & $n-1$  \\
  $\fpp$ & $1$ & $1$ & $3$ & $n+1$ & $n+1$ \\
  $\indfp$ & $1$ & $3$ & $4$ & $n+1$ & $n+1$  \\
  $\indfpp$ & $1$ & $3$ & $4$ & $n+1$ & $n+1$ \\
  $\compfp$ & $2$ & $2$ & $4$ & $2n-2$ & $2n-2$  \\
  $\compfpp$ & $2$ & $2$ & $4$ & $2n-2$ & $2n-2$ \\
  \hline
\end{tabular}
\caption{The value of different thinness variations on crown graphs.}\label{table1} 
\end{center}
\end{table}


We will introduce next some additional definitions and notations that are necessary for this section. Let $G = CR_n$. Define  $\{A_n, A'_n\}$ as the bipartition of $V(G)$ such that $A_n=\{v_1,v_2,\ldots, v_n\}$, $A'_n=\{v'_1,v'_2,\ldots, v'_n\}$ and  $v'_i$ is the only vertex of $A'_n$ that is not adjacent to $v_i$, for all $1 \leq i \leq n$. We will also define $\mirror(v_i)=v'_i$, $\mirror(v'_i)=v_i$, $\side(v_i) = A_n$, and $\side(v'_i)=A'_n$. Given an ordering $<$ of $V(G)$, we say that $v$ is a \textit{little} vertex if $v < \mirror(v)$, and $v$ is a \textit{big} vertex otherwise. Similarly, for a set of vertices $S \subseteq V(G)$, we define $\Little(S)$ as the subset of $S$ that are little vertices.

\subsection{Characterization of consistent solutions for $CR_n$} 

Let $<$ and $\mathcal{V}$ be an ordering and a partition of $V(G)$, respectively. Next, we define two conditions that will be used to characterize consistency of $<$ and $\mathcal{V}$ in $G$.

\begin{condition}\label{cond1:crn:thin}
For any $V \in \mathcal{V}$ and $v \in Little(V)$, $v = \first(V \cap \side(v))$.
\end{condition}

\begin{condition}\label{cond2:crn:thin}
For any $V \in \mathcal{V}$ and $1 \leq i,j \leq n$ such that
 $v_i,v'_j \in V$, there is no $v'_z \in A'_n$ with $z  \neq i$ such that $v_i < v'_j< v'_z$ (resp. $v_z \in A_n$ with $z \neq j$ such that $v'_j < v_i< v_z$). 
\end{condition}

These two conditions are necessary and sufficient to characterize consistency in $CR_n$, as shown in the following theorem.

\begin{theorem}\label{car:crn:thin}
Let $G=CR_n$, and let $<$ and $\mathcal{V}$ be an ordering and a partition of $V(G)$, respectively. 
Then $<$ and $\mathcal{V}$ are consistent for $G$ if, and only if, both Conditions~\ref{cond1:crn:thin} and~\ref{cond2:crn:thin} hold.
\end{theorem}

\begin{proof}
Suppose that $<$ and $\mathcal{V}$ are consistent for $G$. First suppose that Condition~\ref{cond1:crn:thin} does not hold. That is, there are $V \in \mathcal{V}$ and  $v \in V$ such  that $v<\mirror(v)=v'$ and $w=\first(V \cap \side(v)) < v$. There is a contradiction with the fact that $<$ and $\mathcal{V}$ are consistent, as $(w,v')\in E(G)$ and $(v,v') \not \in E(G)$. Now suppose that Condition~\ref{cond2:crn:thin} does not hold. That is, there are $V \in \mathcal{V}$, $v_i,v'_j \in V$  and $z \not \in \{i,j\}$, such that either $v_i < v'_j< v'_z$, or $v'_j < v_i< v_z$. 
In either case, this contradicts the fact that $<$ and $\mathcal{V}$ are consistent, as $(v_i,v'_z) \in E(G)$ and $(v'_j,v'_z) \not \in E(G)$ (resp. $(v'_j,v_z) \in E(G)$ and $(v_i,v_z) \not \in E(G)$).

Consider now that both conditions hold for $<$ and $\mathcal{V}$. Suppose that $<$ is not consistent with $\mathcal{V}$. That is, there are $V \in \mathcal{V}$ and $p< q< r$ such that $p,q \in V$, $(p,r) \in E(G)$ and $(q,r) \not \in E(G)$. Consider, without loss of generality that $\side(p) = A_n$. If $\side(q) = A_n = \side(p)$, then $r = q'$, as $q'$ is the only vertex of $A'_n$ that is not adjacent to $q$.  This fact contradicts Condition~\ref{cond1:crn:thin}, as $q< q'$, and therefore $q$ is little, but $q \neq \first(V \cap \side(q))$, since $p< q$. 
If $\side(q) = A'_n \neq \side(p)$, then $\side(r) = A'_n$, as $r$ is adjacent to $p$. Then there is $v'_z \in A'_n$ with $z  \neq i$ such that $v_i < v'_j< v'_z$, for $v_i = p$, $v'_j = q$, and $v'_z = r$, contradicting Condition~\ref{cond2:crn:thin}.
\end{proof}

\subsection{Consequences for consistent solutions for $CR_n$}

\begin{corollary}
    \label{pitufos-class}
    If $v$ and $w$ are two distinct little vertices of $G=CR_n$ and $\side(v)=\side(w)$, then they cannot belong to the same class in a consistent solution.
\end{corollary}

\begin{proof} Otherwise, it would contradict Condition~\ref{cond1:crn:thin}.
\end{proof}

\begin{corollary}
    \label{last-vertices}
    Let $(\mathcal{V},<)$ be a consistent solution for $G=CR_n$.  If $v$ and $w$ are two vertices of $G$ such that $\side(v)\neq \side(w)$, $v$ and $w$ belong to the same class of $\mathcal{V}$, and $v < w$, 
    then $w$ must be either the last vertex or the penultimate vertex of $\side(w)$. Moreover, if $w$ is the
    penultimate vertex of $\side(w)$, then the last vertex of $\side(w)$ must be $\mirror(v)$.
\end{corollary}
\begin{proof} Otherwise, it would contradict Condition~\ref{cond2:crn:thin}.
\end{proof}

\begin{claim}
    \label{not-pitufo}
    Given four distinct vertices of $G=CR_n$, $v, w$ in one side of the bipartition and $x, y$ in the other side, and such that, in a consistent solution:
    \begin{itemize}
        \item $v$ and $x$ belong to the same class;
        \item $w$ and $y$ belong to the same class;
        \item $v < x$;
        \item $w < y$;
    \end{itemize}
    then at least one of $\{x, y\}$ is not a little vertex.
\end{claim}
\begin{proof}
    Without loss of generality let us assume that $x < y$. By  Corollary~\ref{last-vertices} we know that
    $x$ must be either the last vertex or the penultimate vertex of $\side(x)$; but since $x < y$, we can
    say that $x$ and $y$ are respectively the penultimate and last vertex of $\side(x)$, and that $y=\mirror(v)$. And since $v < x < y$, this implies that $y$ is not a little vertex.
\end{proof}

\begin{claim}\label{cr-at-least}
In a consistent solution for $G=CR_n$, the number of classes containing a little vertex is at least $n-1$.
In other words, at most one class can contain two little vertices, one of each side, since by Corollary~\ref{pitufos-class} no class can contain more than one little vertex of the same side.
\end{claim}
\begin{proof}
Suppose, on the contrary, that there are 4 distinct little vertices $v$, $w$, $x$ and $y$ such as $v,w \in A_n$ and $x,y \in A'_n$,
$v$ and $x$ belong to the same class, $w$ and $y$ belong to the same class (by Corollary~\ref{pitufos-class}, no class can contain more than one little vertex of the same side).
By Claim \ref{not-pitufo}, it is neither possible that $v<x$ and $w<y$, nor that $v>x$ and $w>y$.

Suppose then, without loss of generality, that $v < x$ and $w > y$. By Corollary~\ref{last-vertices},
$w$ and $x$ can only be the last or penultimate vertices in $A_n$ and $A'_n$, respectively. Moreover, if $w$ is the penultimate vertex of $A_n$ then $\last(A_n)=\mirror(y)$, 
    and if $x$ is the
    penultimate vertex of $A'_n$ then $\last(A'_n)=\mirror(v)$.
Without loss of generality, let us assume $w > x$. Notice as $\mirror(w) > w > x$
the vertex $x$ cannot be the last vertex in $A'_n$, implying that $x$ is the penultimate vertex in $A'_n$ and $\mirror(w) = \last(A'_n)$. 
But then $\mirror(w) = \last(A'_n)=\mirror(v)$; thus $w = v$, a contradiction.

By the arguments above, there is at most one pair of little vertices $\{v, x\}$ such that they belong to the same class.
Then there must be at least
$$|Little(A_n)| + |Little(A'_n)| - 1 = n-1$$
different classes containing the little vertices. 
\end{proof}

\subsection{Consequences for strongly consistent solutions for $CR_n$}

\begin{corollary}
    \label{pitufos-class-st}
    If $v$ and $w$ are two distinct big vertices of $G=CR_n$ and $\side(v)=\side(w)$, then they cannot belong to the same class in a strongly consistent solution.
\end{corollary}

\begin{proof} Otherwise, the reverse order would contradict Corollary~\ref{pitufos-class}, since big vertices for an ordering are little vertices for its reverse.
\end{proof}

\begin{corollary}
    \label{no-three-side}
    No class contains three vertices of the same side in a strongly consistent solution of $G=CR_n$.
\end{corollary}

\begin{proof} Since every vertex is either little or big, if a class contains three vertices of the same side then either there are at least two little vertices of the same side, contradicting  Corollary~\ref{pitufos-class}, or there are at least two big vertices of the same side, contradicting  Corollary~\ref{pitufos-class-st}.
\end{proof}

\begin{corollary}
    \label{first-vertices}
    Let $(\mathcal{V},<)$ be a strongly consistent solution for $G=CR_n$.  If $v$ and $w$ are two vertices of $G$ such that $\side(v)\neq \side(w)$, $v$ and $w$ belong to the same class of $\mathcal{V}$, and $v < w$, 
    then $v$ must be either the first or second vertex of $\side(v)$. Moreover, if $v$ is the
    second vertex of $\side(v)$, then the first vertex of $\side(v)$ must be $\mirror(w)$.
\end{corollary}
\begin{proof} It follows from applying Corollary~\ref{last-vertices} to the reverse order of $<$.
\end{proof}

\begin{claim}
    \label{consec-side}
    Let $(\mathcal{V},<)$ be a strongly consistent solution for $G=CR_n$.  If $n \geq 4$, for every class $V \in \mathcal{V}$ the vertices of each side are consecutive in the order restricted to $V$. 
\end{claim}
\begin{proof} Suppose, on the contrary, that in some class $V$ there are three vertices $a < b < c$ such that $\side(a) = \side(c) \neq \side(b)$. Since $n \geq 4$, there is $d \in V(G)$ such that $\side(d) = \side(b)$, and $d \not \in \{b, \mirror(a), \mirror(c)\}$. Then $\{a,c\} \subseteq N[d]$ and $b \not \in N[d]$, which contradicts Lemma~\ref{prop:neigh:propthin}.
\end{proof}

\begin{corollary}
    \label{no-four}
    No class contains four vertices in a strongly consistent solution of $G=CR_n$, $n \geq 4$.
\end{corollary}

\begin{proof} 
Let $(\mathcal{V},<)$ be a strongly consistent solution for $G=CR_n$, where $n \geq 4$. Suppose a class $V \in \mathcal{V}$ contains at least four vertices. By Corollary~\ref{no-three-side}, it contains exactly four vertices, two of each side. Moreover, by 
Claim~\ref{consec-side}, the class must be composed by four distinct vertices $a < b < c < d$ such that $\side(a)=\side(b)$, $\side(c)=\side(d)$, and $\side(a)\neq \side(c)$.

Applying Corollary~\ref{last-vertices} to $a < c$ and $b < c$, it follows that $c$ must be the penultimate vertex of $\side(c)$, thus $d = \mirror(a)$ and $d=\mirror(b)$, a contradiction. 
\end{proof}

\begin{claim}\label{cr-at-least-st}
In a strongly consistent solution for $G=CR_n$, the number of classes containing a big vertex is at least $n-1$.
In other words, at most one class can contain two big vertices, one of each side, since by Corollary~\ref{pitufos-class-st} no class can contain more than one big vertex of the same side.
\end{claim}

\begin{proof}
It follows from Claim~\ref{cr-at-least} applied to the reverse order of the solution.
\end{proof}

\begin{claim}
    \label{no-two-three}
    There are at most two classes of three vertices in a strongly consistent solution of $G=CR_n$, $n \geq 4$. Moreover, if there are two such classes, then one of them has vertices $(a,b,c)$, where $a$ and $b$ are the first two vertices of their side and $c$ belongs to the opposite side, and the other one has vertices $(x,y,z)$, where $y$ and $z$ are the last two vertices of the side of $a$ and $x$ belongs to the side of $c$. In particular,   
    they contain four vertices of one side and two of the other side, and the four vertices of the same side are the first two and the last two of the side in the order.
\end{claim}

\begin{proof} 
Let $(\mathcal{V},<)$ be a strongly consistent solution for $G=CR_n$, where $n \geq 4$. Suppose a class $V \in \mathcal{V}$ contains three vertices. Then it contains either two little vertices or two big vertices. By Claims~\ref{cr-at-least} and~\ref{cr-at-least-st}, this may happen only once for little vertices and once for big vertices, so there are at most two classes of three vertices. 
Suppose there are indeed two classes $V, W$ of three vertices each. By Corollary~\ref{no-three-side} and Claim~\ref{consec-side}, we may assume that $V = \{a,b,c\}$ where $a < b < c$, $\side(a)=\side(b)$ and $\side(b)\neq \side(c)$. By Corollary~\ref{first-vertices} applied to $b < c$, it follows that $a$ and $b$ are the first and second vertices of $\side(a)$, and $a = \mirror(c)$. In particular, $a$ is little and $c$ is big. By Corollary~\ref{pitufos-class}, $b$ has to be big too.    
Suppose $W = \{x,y,z\}$, $x < y < z$, has two vertices of $\side(c)$ and one of $\side(a)$. By Corollary~\ref{no-three-side} and Claim~\ref{consec-side}, there are two cases: either $\side(x)=\side(y)=\side(c)$, and $\side(z)=\side(a)$, or $\side(y)=\side(z)=\side(c)$, and $\side(x)=\side(a)$. In the first case, reasoning as above, $x$ and $y$ are the first and second vertices of $\side(c)$, $x = \mirror(z)$, $x$ is little and $y, z$ are big. This is a contradiction to Claim~\ref{cr-at-least-st}. In the second case, by Corollary~\ref{first-vertices} applied to $x < y$, it follows that $x$ is either the first or the second vertex of $\side(a)$, but $x \not \in \{a,b\}$, a contradiction. 
Thus, suppose $W = \{x,y,z\}$, $x < y < z$, has two vertices of $\side(a)$ and one of $\side(c)$. By Corollary~\ref{no-three-side} and Claim~\ref{consec-side}, there are two cases: either $\side(x)=\side(y)=\side(a)$, and $\side(z)=\side(c)$, or $\side(y)=\side(z)=\side(a)$, and $\side(x)=\side(c)$. In the first case, reasoning as above, $x$ and $y$ are the first and second vertices of $\side(a)$, a contradiction because $x \neq a$. So the second case holds, and $y$ and $z$ are the penultimate and last vertices of $\side(a)$. 
\end{proof}

\subsection{Proofs of the results in Table~\ref{table1}}

\begin{theorem}\label{thin-crn}
For $n \geq 1$, $\thin(CR_n) = \max(1,n-1)$.
For $n$ odd, $n \neq 3$, $\pthin(CR_n) = n$; for $n=3$ or $n \geq 1$ even, $\pthin(CR_n) = n-1$. 
\end{theorem}
\begin{proof}
The thinness of a non-empty graph is always at least~1, and by Claim~\ref{cr-at-least}, $\thin(CR_n) \geq n-1$.

Suppose $n\geq 5$ is odd, and suppose $(\mathcal{V},<)$ is a strongly consistent solution for $G=CR_n$, with $n-1$ classes. By Corollary~\ref{no-four} and Claim~\ref{no-two-three}, the solution has to have two classes of three vertices each, and $n-3$ classes of two vertices each. Moreover, since for one of the sides, the first two and the last two vertices belong to these classes of size three, by Corollary~\ref{last-vertices} and Corollary~\ref{first-vertices}, no other class can contain elements of both sides. 
But since $n$ is odd, both $n-4$ and $n-2$ are odd, so there is no way to partition the remaining $2n-6$ vertices into classes of size $2$ where no class can contain elements of both sides. Hence, for $n \geq 5$ odd, $\pthin(CR_n) \geq n$.  

We will conclude the proof by construction. We will first show strongly partitions and orders for $n \leq 4$. 

For $n=1$, let $\sigma = (v_1,v'_1)$ and $\mathcal{V} = \{\{v_1,v'_1\}\}$.

For $n=2$, let $\sigma = (v_1,v'_2,v_2,v'_1)$ and $\mathcal{V} = \{\{v_1,v'_2,v_2,v'_1\}\}$.

For $n=3$, let $\sigma = (v_1,v'_2,v'_3,v_3,v_2,v'_1)$ and $\mathcal{V} = \{\{v_1,v'_3,v_2\},\{v'_2,v_3,v'_1\}\}$.

For $n=4$, let $\sigma = (v_1,v'_3,v_4,v'_4,v'_2,v_2,v_3,v'_1)$ and $\mathcal{V} = \{\{v_1, v'_2, v'_1\},$ $\{v'_3, v'_4, v_3\},$ $\{v_4, v_2\}\}$.

In all the cases, it is not hard to check that Condition~\ref{cond1:crn:thin} and Condition~\ref{cond2:crn:thin} are satisfied for both $(\mathcal{V}, \sigma)$ and $(\mathcal{V}, \overline{\sigma})$. 

For $n\geq 6$ even, we build a strongly consistent solution $(\mathcal{V},\sigma)$  with $n-1$ classes for $CR_n$, from the solution for $n=4$ by inserting the remaining vertices into the order and distributing them into $n-4$ classes of two vertices each. The construction is shown in Algorithm~\ref{even-crn}.

For $n \geq 5$ odd, we can build a strongly consistent solution $(\mathcal{V},\sigma)$  with $n$ classes for $CR_n$ following Algorithm~\ref{odd-crn-st}, and a consistent solution $(\mathcal{V},\sigma)$  with $n-1$ classes for $CR_n$ following Algorithm~\ref{odd-crn}. 

In each case, it is not hard to check that Condition~\ref{cond1:crn:thin} and Condition~\ref{cond2:crn:thin} hold either for  $(\mathcal{V}, \sigma)$ or for both 
 $(\mathcal{V}, \sigma)$ and $(\mathcal{V}, \overline{\sigma})$, as required. \end{proof}

\begin{algorithm}
\caption{Strongly consistent layout of $CR_n$ for even $n \geq 6$.}
\label{even-crn}
\begin{algorithmic}
    \Function{StConsistentLayoutCR}{%
    $v_1, v_2, \dots v_n$: vertices in $A_n$,
    $v'_1, v'_2, \dots v'_n$: vertices in $A'_n$
}
\State $C_{1} \gets \{v_1, v'_2, v'_1\}$
\State $C_{2} \gets \{v'_3, v'_4, v_3\}$
\State $C_{3} \gets \{v_4, v_2\}$
\State $C_{4} \gets \{v_5, v_{n}\}$
\State $C_{5} \gets \{v'_{6}, v'_{5}\}$

\State \( \sigma \gets (v'_6, v'_5) \) 
 
\For{odd $i \in [7..n-1]$} \\
    \Comment{in each iteration we set the order and classes of 4 vertices}
    \State $\sigma \gets \sigma \oplus (v_i, v_{i-1}, v'_{i+1}, v'_i)$ 

    \State $C_{i-1} \gets \{v_i, v_{i-1}\}$
    \State $C_{i} \gets \{v'_{i+1}, v'_{i}\}$
\EndFor
\State \( \sigma \gets (v_1,v'_3,v_4,v'_4,v_5) \oplus \sigma \oplus (v_n,v'_2,v_2,v_3,v'_1)\) \\
\Comment{we set the order of the remaining vertices}
\State \( \mathcal{V} \gets \{ C_1,\dots,C_{n-1} \}\) \\
\Return $\sigma, \mathcal{V}$
\EndFunction

\end{algorithmic}
\end{algorithm}

\begin{algorithm}
\caption{Consistent layout of $CR_n$  for odd $n > 1$.}
\label{odd-crn}
\begin{algorithmic}
    \Function{ConsistentLayoutCR}{%
    $v_1, v_2, \dots v_n$: vertices in $A_n$,
    $v'_1, v'_2, \dots v'_n$: vertices in $A'_n$
}

\State \( \sigma \gets \emptyset \) \Comment{the ordering starts out empty}
\For{odd $i \in [1..n-2)$} \\
    \Comment{in each iteration we set the order and classes of 4 vertices}
    \State $\sigma \gets \sigma \oplus (v_i, v'_{i+1}, v_{i+1}, v'_i)$ 

    \State $C_{i} \gets \{v_i, v_{i+1}\}$
    \State $C_{i+1} \gets \{v'_{i+1}, v'_{i}\}$
\EndFor
\State \( \sigma \gets (v_{n}) \oplus \sigma \oplus (v'_{n}) \)
\Comment{we set the order and class of the remaining vertices}
\State \( C_{n-1} \gets C_{n-1} \cup \{v_n\} \)
\State \( C_{n-2} \gets C_{n-2} \cup \{v'_n\} \)
\State \( \mathcal{V} \gets \{ C_1,...,C_{n-1} \}\)\\
\Return $\sigma, \mathcal{V}$
\EndFunction

\end{algorithmic}
\end{algorithm}

\begin{algorithm}
\caption{Strongly consistent layout of $CR_n$  for odd $n > 4$.}
\label{odd-crn-st}
\begin{algorithmic}
    \Function{StConsistentLayoutCR}{%
    $v_1, v_2, \dots v_n$: vertices in $A_n$,
    $v'_1, v'_2, \dots v'_n$: vertices in $A'_n$
}

\State \( \sigma \gets \emptyset \) \Comment{the ordering starts out empty}
\For{even $i \in [2..n)$} \\
    \Comment{in each iteration we set the order and classes of 4 vertices}
    \State $\sigma \gets \sigma \oplus (v'_i, v'_{i-1}, v_{i+1}, v_i)$ 

    \State $C_{i} \gets \{v'_i, v'_{i-1}\}$
    \State $C_{i+1} \gets \{v_{i+1}, v_i\}$
\EndFor
\State \( \sigma \gets (v_{1}) \oplus \sigma \oplus (v'_{n}) \)
\Comment{we set the order and class of the remaining vertices}
\State $C_{1} \gets \{v_1, v'_{n}\}$
\State \( \mathcal{V} \gets \{ C_1,...,C_{n} \}\)\\
\Return $\sigma, \mathcal{V}$
\EndFunction

\end{algorithmic}
\end{algorithm}





Next, we deal with the independent and complete versions of thinness in crowns. 

\begin{theorem}\label{thm:crn:indthin}
For every $n \geq 1$, $\indthin(CR_n) =  \indpthin(CR_n) = n$.
\end{theorem}

\begin{proof}
Let $G =CR_n$. We prove that $\indthin(G) \geq n$ and $\indpthin(G) \leq n$.

Let $<$ and $\mathcal{V}$ be an ordering and a partition of $V(G)$ into independent sets, respectively, that are consistent. Let $V \in \mathcal{V}$, and $v \in V$. Since $V$ is an independent set of $G$, either $V=\{v,\mirror(v)\}$ or $V \subseteq \side(v)$. Thus, by  Condition~\ref{cond1:crn:thin}, if $v$ is a little vertex then $v = \first(V)$. This implies that for each pair $\{v_i,v'_i\}$ there is a class $V_i \in \mathcal{V}$ in which either $v_i$ or $v'_i$ is the first vertex of $V_i$. In particular, there is at least one class in $\mathcal{V}$ for each pair of vertices $\{v_i,v'_i\}$ in $G$, in other words, $|\mathcal{V}| \geq n$.

Now we prove by induction that there is an ordering  $\sigma_n$, and an $n$-partition $\mathcal{V}$, of $V(G)$ such that both $\sigma_n$ and $\overline{\sigma_n}$ respect  Condition~\ref{cond1:crn:thin} and  Condition~\ref{cond2:crn:thin}. Moreover, if $n$ is even, then  each class of  $\mathcal{V}$ is composed by either vertices of $A_n$ or $A'_n$. The case where $n=1$ is trivial, $\mathcal{V}_1 = \{\{v_1,v'_1\}\}$ and $\sigma_1 = (v_1,v'_1)$. For $n \geq 2$, define $\mathcal{V}^2_n =\{\{v_{n-1},v_n\},\{v'_{n-1},v'_n\}\}$ and $\sigma^2_n = (v_{n-1}, v'_n, v_n, v'_{n-1})$. One can easily verify that both $\sigma^2_2$ and $\overline{\sigma^2_2}$ are in accordance with  Condition~\ref{cond1:crn:thin} and  Condition~\ref{cond2:crn:thin} with respect to $\mathcal{V}^2_2$. Suppose that the claim holds for all $n' < n$. There are  two cases to consider.

If $n$ is even, then define $\mathcal{V} =\mathcal{V}_{n-2} \cup \mathcal{V}^2_n$ and $\sigma_n = \sigma_{n-2}\oplus\sigma^2_n$. By the induction hypothesis,  Condition~\ref{cond2:crn:thin} does not apply to  $\mathcal{V}$ and   Condition~\ref{cond1:crn:thin} holds for $\mathcal{V}_{n-2}$ with the orderings $\sigma_{n-2}$ and $\overline{\sigma_{n-2}}$, and also for $\mathcal{V}^2_{n}$ with the orderings $\sigma^2_{n}$ and $\overline{\sigma^2_{n}}$. Therefore, $\mathcal{V}_{n}$ and $\sigma_n$ are strongly consistent. 

If $n$ is odd, then define  $\mathcal{V} =\mathcal{V}_{n-1} \cup \{\{v_{n},v'_{n}\}\}$ and $\sigma_n = (v_n)\oplus\sigma_{n-1}\oplus(v'_n)$. Regarding the induction hypothesis, it is possible to conclude two things. First,  Condition~\ref{cond1:crn:thin} holds for both $\sigma_{n-1}$ and $\overline{\sigma_{n-1}}$. Second,  Condition~\ref{cond2:crn:thin} only applies to the class $\{v_{n},v'_{n}\}$. Since all the neighbors of $v_n$ and $v'_n$ are in between them in $\sigma_n$, then both $\sigma_n$ and $\overline{\sigma_n}$ are in accordance with  Condition~\ref{cond1:crn:thin} and  Condition~\ref{cond2:crn:thin}. Hence, $\mathcal{V}_{n}$ and $\sigma_n$ are strongly consistent.
\end{proof}

\begin{theorem}\label{comp-thin-crn}
For $n\geq 4$, $\compthin(CR_n) = 2n-4$, while $\compthin(CR_3) = 3$, and $\compthin(CR_1) = \compthin(CR_2) = 2$. The same holds for the proper version.
\end{theorem}

\begin{proof}
Let $n \geq 4$. In a partition of $CR_n$ into completes, the classes have either size $1$ or size $2$, and classes of size $2$ are $\{v_i,v_j'\}$ with $i \neq j$. Suppose there are $r$ classes of size $2$, and, by symmetry, suppose without loss of generality that the greatest vertex of the order is $v_n'$. Then for every class $\{v_i,v_j'\}$ of size $2$, $j\neq n$, except perhaps the class where $i=n$, we have $v_j' < v_i$ (by Condition~\ref{cond2:crn:thin}). So, at least $r-2$ classes $\{v_i,v_j'\}$ satisfy $v_j' < v_i$. Let $v_k$ be the greatest vertex among the ones in those $r-2$ classes. If $r-2 > 2$, there is a class $\{v_i,v_j'\}$ satisfying $v_j' < v_i$ and such that $i \neq k$ and $j \neq k$. So $v_j' < v_i < v_k$ and the indices are all different, contradicting Condition~\ref{cond2:crn:thin}. Then $r-2 \leq 2$, so $r \leq 4$, and $\compthin(CR_n) \geq 2n-4$. 

A complete proper thin representation satisfying the bound is the following:
$\sigma=(v'_1, v'_2, v'_5, \dots, v'_n, v_4, v_3,  v_5, \dots, v_n, v_2, v_1, v'_3, v'_4)$ where the classes of size $2$ are $\{v'_1,v_2\}$, $\{v'_2,v_1\}$, $\{v_4,v'_3\}$, $\{v_3,v'_4\}$.
For $n=1$, we need two classes, $\mathcal{V}=\{\{v_1\},\{v'_1\}\}$. For $n = 2$, $\mathcal{V}=\{\{v_1,v'_2\},\{v_2,v'_1\}\}$. In both cases, any order works. For $n=3$, the order $\sigma=(v_1, v'_2, v_3, v'_1, v_2, v'_3)$ and the classes $\mathcal{V}=\{\{v_1,v'_3\}, \{v'_2,v_3\}, \{v'_1,v_2\}\}$ are a complete proper thin representation.
\end{proof}

Now we study the precedence thinness and its variations in crowns. 

\begin{theorem}
For $n\geq 1$, $\fp(CR_n) = \max(1,n-1)$.
\end{theorem}

\begin{proof}
By Theorem~\ref{thin-crn}, $\fp(CR_n) \geq \max(1,n-1)$. We prove constructively that $\fp(CR_n) \leq \max(1,n-1)$. 

For $n=1$, $V_1 = (v_1,v'_1)$. Since we are dealing with precedence thinness, the order of the vertices is implied by the internal order of each of the classes which are, for $n \geq 2$, in order, $\mathcal{V} = (V_1,V_2,\ldots,V_{n-1})$. 

For $n=2$, let $V_1 = (v_1,v'_2,v_2,v'_1)$.

For $n=3$, let $V_1 = (v'_1)$, $V_2 = (v_2,v'_3,v_1,v'_2,v_3)$.

For $n=4$ let $V_1 = (v'_1)$, $V_2 = (v_2,v_1)$, $V_3 = (v'_3,v_4,v'_2,v_3,v'_4)$.

The consistency for these cases is easy to check. 

For $n > 4$ let $V_1 = (v'_1)$, $V_2 = (v_2,v_1)$, $V_3 = (v'_3,v'_2)$; for $3 < i < n-1$, $V_i = (v_i,v_{i-1})$ if $i$ is even
and $V_i = (v'_i,v'_{i-1})$ if $i$ is odd; $V_{n-1} = (v'_{n-1},v_n,v'_{n-2},v_{n-1},v'_n)$ if $n$ is even, and $V_{n-1} = (v_{n-1},v'_n,v_{n-2},v'_{n-1},v_n)$ if $n$ is odd. 

Note that, for all $V \in \mathcal{V}$, $G[V]$ is an interval graph and the order is a canonical order of $G[V]$. It is easy to see that both Conditions~\ref{cond1:crn:thin} and~\ref{cond2:crn:thin} are being satisfied  for $V_1$ and $V_{n-1}$. Each one of the remaining parts consists of vertices of either $A_n$ or $A'_n$, where the first vertex is a little vertex and the second one is a big vertex. Thus, for these remaining parts, both conditions are also satisfied. 
Therefore, the orders and partitions defined are consistent.
\end{proof}






\begin{theorem}\label{fpp-crn}
For $n\geq 4$, $\fpp(CR_n) = n+1$, while $\fpp(CR_1) = \fpp(CR_2) = 1$, and $\fpp(CR_3) = 3$. 
\end{theorem}

\begin{proof}
The cases of $CR_1$ and $CR_2$ are trivial, since they are proper interval graphs. Let $n \geq 3$, and suppose $(\mathcal{V},<)$ is a strongly consistent solution for $G=CR_n$. 

We claim that if $V$ is the first or the last part in the solution, then $|V|=1$. Let $V$ be the first part, and suppose by contradiction that $|V|>1$. Let $a < b \in V$ be the first two vertices of $V$. If $\side(a)=\side(b)$, then they are two little vertices of the same side in the same class, contradicting Corollary~\ref{pitufos-class}. 
If $\side(a)\neq\side(b)$, then, by Corollary~\ref{last-vertices}, $b$ is either the penultimate or the last vertex of $\side(b)$, a contradiction because it is the first vertex of $\side(b)$ and there are at least three vertices on its side. 
The argument for the last class is analogous, by using Corollary~\ref{pitufos-class} and Corollary~\ref{first-vertices}. 

In particular, this implies that $\fpp(CR_3) \geq 3$. A strongly consistent solution with three classes is $(v_1)(v'_2,v_3,v'_1,v_2)(v'_3)$. 

Suppose now $n \geq 4$. By Corollary~\ref{no-four}, every class has size at most three. Moreover, given that the first class and the last class have size one, in order to obtain a solution with less than $n+1$ classes, at least two classes of size three are needed. In that case, by Claim~\ref{no-two-three}, 
one of them, say $V$, has vertices $(a,b,c)$, where $a$ and $b$ are the first two vertices of their side, say $A_n$, and $c$ belongs to $A'_n$, and the other one, say $W$, has vertices $(x,y,z)$, where $y$ and $z$ are the last two vertices $A_n$ and $x$ belongs to $A'_n$. 
 In particular, $V$ precedes $W$, but these are not the first or last classes, because the first and the last class contain only one vertex. Indeed, this situation implies that the only vertex $v$ in the first class and the only vertex $w$ in the last class belong to $A'_n$. But then we have $c < x < w \in A'_n$, which contradicts Corollary~\ref{last-vertices} applied to $b < c$. Hence, for $n \geq 4$, $\fpp(CR_n) \geq n+1$.    

We will complete the proof by construction. Define the following ordered sets: 
\begin{itemize}
    \item $V_1 = (v_1)$ 
    
    \item for all $2\leq i \leq n$, if $i$ is even, then  $V_i = (v'_i,v'_{i-1})$,  otherwise,  $V_i = (v_i,v_{i-1})$

     \item if $n$ is even, then $V_{n+1} = (v_n)$, otherwise,  $V_{n+1} = (v'_n)$ 
\end{itemize}

Let $\mathcal{V} = (V_1,V_2, \ldots, V_n, V_{n+1})$ be an ordered  $(n+1)$-partition of $V(CR_n)$.
Note that, for all parts $V_i = \{v_i,v_{i-1}\}$ (resp. $V_i = \{v'_i,v'_{i-1}\}$), $v_i < v'_i$ and $v_{i-1}>v'_{i-1}$ (resp.  $v'_i < v_i$ and $v'_{i-1}>v_{i-1}$). Thus, Condition~\ref{cond1:crn:thin} is satisfied for both the order and its reversal. Moreover, note that all parts consist of vertices of either $A_n$ or $A'_n$, so Condition~\ref{cond2:crn:thin} is also satisfied for both the order and its reversal. 
Therefore, the order defined is strongly consistent with $\mathcal{V}$, and for $n \geq 4$, $\fpp(CR_n) \leq n+1$.   
\end{proof}

\begin{theorem}\label{thm:crn:indthin-prec}
For $n\geq 2$, $\indfp(CR_n) = n+1$, and $\indfp(CR_1) = 1$. The same holds for the proper version.
\end{theorem}

\begin{proof}
Let $G =CR_n$, $n \geq 2$. We prove that $\indfp(G) \geq n+1$ and $\indfpp(G) \leq n+1$.

Let $<$ and $\mathcal{V}$ be an ordering and a precedence partition of $V(G)$ into independent sets, respectively, that are consistent. Let $V \in \mathcal{V}$, and $v \in V$. Since $V$ is an independent set of $G$, either $V=\{v,\mirror(v)\}$ or $V \subseteq \side(v)$. Thus, by  Condition~\ref{cond1:crn:thin}, if $v$ is a little vertex then $v = \first(V)$. This implies that for each pair $(v_i,v'_i)$ there is a class of the partition in which either $v_i$ or $v'_i$ (the little one) is the first vertex of the class.

By the precedence constraint and Condition~\ref{cond2:crn:thin}, at most one class is composed by a vertex and its mirror. Since $n \geq 2$, there is a little vertex that is not in the same class as its mirror. Let $v$ be the greatest such vertex according to $<$, and let $V \in \mathcal{V}$ such that $\mirror(v) \in V$. Then $v < \mirror(v)$, $v \not \in V$, and by the precedence constraint, $v < \first(V)$. Since $\mirror(\first(V)) \neq \mirror(v)$, $\mirror(\first(V))$ is not in the same class as its mirror. By definition of $v$, $\first(V)$ is not a little vertex. Therefore, there is a class of the partition for each little vertex $w$, where $w$ is the first of the class, plus the class $V$ whose first vertex is not little. Hence $\indfp(G) \geq n+1$.  

An order and precedence partition into $n+1$ independent sets that are strongly consistent can be defined as follows.

For $n = 2$, 
$(v_1) (v'_2,v'_1) (v_2)$.  

For $n \geq 3$ odd, 
$(v_1) [(v'_{i+1},v'_i) (v_{i+2},v_{i+1})]_{1 \leq i \leq n-2}  (v'_n)$.

For $n \geq 4$ even, 
$(v_1) [(v'_{i+1},v'_i) (v_{i+2},v_{i+1})]_{1 \leq i \leq n-3}(v'_n,v'_{n-1})(v_n)$. 
\end{proof}

\begin{theorem}\label{prec-comp-thin-crn}
For $n\geq 2$, $\compfp(CR_n) = 2n-2$, and $\compfp(CR_1) = 2$. The same holds for the proper version.
\end{theorem}

\begin{proof}
Let $n \geq 2$. In a partition of $CR_n$ into completes, the classes have either size $1$ or size $2$, and classes of size $2$ are $\{v_i,v_j'\}$ with $i \neq j$. Suppose that $\{v_i,v_j'\}$ is a class in a precedence partition consistent with a vertex order, so $v_i$ and $v'_j$ are consecutive in the order, and without loss of generality, $v_i < v'_j$.
Suppose that there is another class $\{v_a,v'_b\}$ of size two in the partition. Then either both are greater than $v'_j$, or both are smaller than $v_i$. In the first case, by Condition~\ref{cond2:crn:thin}, it must be $b = i$. 
In the second case, either $v_a < v'_b$ in which case $a = j$ by Condition~\ref{cond2:crn:thin}; or $v'_b < v_a$ in which case $b = i$ by Condition~\ref{cond2:crn:thin}.
In particular, in any case, either $b=i$ or $a=j$. So, there are at most two more classes of size two. 

Suppose that there are two such classes. By the observations above, one of them contains $v'_i$ and the other one contains $v_j$. Namely, the classes are $\{v'_i,v_k\}$ and $\{v'_\ell,v_j\}$, and both $v'_\ell$ and $v_j$ are smaller than $v_i$, since $i \neq \ell$. By the same observations, since $i \neq \ell$, then $v_j < v'_\ell$. 
Doing the same analysis as before but for the class $\{v'_i,v_k\}$ with respect to the class $\{v_j,v'_\ell\}$, since $i\neq j$, $v'_i$ and $v_k$ are smaller than $v_j$, $v_k < v'_i$, and $k=\ell$. 
But then, $v_k < v'_i < v'_j$ with $k \neq j$, a contradiction to Condition~\ref{cond2:crn:thin}.

Therefore, there are at most two classes of size two in a precedence partition consistent with a vertex order, and $\compfp(CR_n) \geq 2n-2$. 

A complete proper precedence thin representation satisfying the bound is the following:
$\sigma=(v'_3, \dots, v'_n, v_1, v'_2, v_2, v'_1, v_3,  \dots, v_n)$ where the classes of size $2$ are $\{v_1,v'_2\}$ and $\{v_2,v'_1\}$.
For $n=1$, two classes are needed because the graph is edgeless. \end{proof}

\section{Precedence thinness of cographs} \label{sec:cographs}

In this section, we describe the behavior of the precedence thinness under the union and join operations, which allows to compute the precedence thinness of cographs efficiently.

With respect to the behavior of the thinness under the union and join operations, the following results were proved in the literature. These results allowed to compute the thinness of cographs in polynomial time, as well as to characterize $k$-thin graphs by forbidden induced subgraphs within the class of cographs. 

\begin{theorem}[\hspace{1sp}\cite{B-D-thinness}] \label{thm:union}
Let $G_1$ and $G_2$ be graphs. Then, $\thin(G_1 \cup G_2) =
\max(\thin(G_1),$ $\thin(G_2))$.
\end{theorem}

\begin{theorem}[\hspace{1sp}\cite{BGOSS-thin-oper}] \label{thm:join2}
Let $G_1$ and $G_2$ be graphs. If $G_2$ is complete, then
$\thin(G_1 \vee G_2) = \thin(G_1)$. If neither $G_1$ nor $G_2$ are
complete, then $\thin(G_1 \vee G_2) = \thin(G_1) + \thin(G_2)$.
\end{theorem}

Also in \cite{BGOSS-thin-oper}, it was observed that the proper thinness of the join $G_1 \vee G_2$ cannot be expressed as a function whose only parameters are the proper thinness of
$G_1$ and $G_2$ (even excluding simple particular cases, like
trivial or complete graphs).

Next, we present a theorem that describes the precedence thinness over the union of two graphs.

\begin{theorem}[Precedence thinness of union]\label{thm:fp-union}
Let $G_1$ and $G_2$ be graphs. Then, $\fp(G_1 \cup G_2) =
\fp(G_1)+\fp(G_2)-1$.
\end{theorem}

\begin{proof}
Let $\sigma_1$ be an order of $V(G_1)$, and let $V_1, \dots, V_{k_1}$ be a precedence partition of $V(G_1)$, consistent with $\sigma_1$, such that $k_1 = \fp(G_1)$. Define $\sigma_2$ and $W_1, \dots, W_{k_2}$ in the same way for $G_2$. It is not difficult to see that $V_1, \dots, V_{k_1} \cup W_1, \dots, W_{k_2}$ is a precedence partition of $V(G_1 \cup G_2)$ into $k_1+k_2-1$ sets, consistent with the order $\sigma_1 \oplus \sigma_2$. So $\fp(G_1 \cup G_2) \leq
\fp(G_1)+\fp(G_2)-1$. 

Now let $\sigma$ be an order of $V(G_1 \cup G_2)$, and let $\mathcal{V} = \{V_1, \dots, V_{k}\}$ be a precedence partition of $V(G_1 \cup G_2)$, consistent with $\sigma$, such that $k = \fp(G_1 \cup G_2)$. We call a class \emph{mixed} if it contains vertices of both $G_1$ and $G_2$.
We will prove that $G_1 \cup G_2$ admits a vertex ordering and a consistent precedence partition into the same number $k$ of classes, such that at most one class is mixed. Moreover, the order restricted to each of $V(G_1)$ and $V(G_2)$ coincides with $\sigma$.

Suppose $\mathcal{V}$ contains more than one mixed class, and let $i$ be the smallest index such that $V_i$ is mixed. Notice that $i < k$. Let $V_i^1 = V_i \cap V(G_1)$ and $V_i^2 = V_i \cap V(G_2)$. Suppose that $\last(V_i)$ belongs to $V(G_1)$. Then, for every $j > i$, the vertices of $V_i^2$ have no neighbors in $V_j$: they have no neighbors in $V(G_1) \cap V_j$ by definition of disjoint union, and they cannot have a neighbor in $V(G_2) \cap V_j$, since such a vertex should be adjacent to $\last(V_i)$ by consistency, and that contradicts the definition of disjoint union. 
We will define a partition $\mathcal{V'}= [\mathcal{V} \setminus \{V_i,V_{i+1}\}] \cup \{V_i^1, V_i^2 \cup V_{i+1}\}$, and the order $\sigma'$ such that 
$v <_{\sigma'} w$ if, and only if, either $v \in V_i^1$ and $w \in V_i^2$, or 
$v <_{\sigma} w$ and it is not the case that $v \in V_i^2$ and $w \in V_i^1$.

Let $x <_{\sigma'} y <_{\sigma'} z$ such that $x$ and $y$ are in the same class of $\mathcal{V'}$, $x$ and $z$ are adjacent. Since $\sigma$ and $\mathcal{V}$ are consistent, if $x <_{\sigma} y <_{\sigma} z$ and $x$ and $y$ are in the same class of $\mathcal{V}$, then $z$ is adjacent to $y$. So, suppose first that $x$ and $y$ are not in the same class of $\mathcal{V}$. Then $x \in V_i^2$ and $y \in V_{i+1}$. This implies that $z \in V_j$ with $j \geq i+1$, but $x$ and $z$ are adjacent, and vertices of $V_i^2$ have no neighbors in $V_j$ with $j > i$, a contradiction. Suppose then that some of the two inequalities do not hold for $\sigma$. By definition of $\sigma'$, this can only happen if one of the vertices belongs to $V_i^1$ and the other one to $V_i^2$. Since $x$ and $y$ are in the same class of $\mathcal{V'}$, it must be $y \in V_i^1$ and $z \in V_i^2$, implying that $x \in V_i^1$. But this contradicts the fact that $x$ and $z$ are adjacent.  
We conclude that $z$ is adjacent to $y$, and that  $\sigma'$ and $\mathcal{V'}$ are consistent.

The case in which $\last(V_i)$ belongs to $V(G_2)$ is symmetric, and we can repeat this procedure until obtaining an ordering and a consistent partition into $k$ classes, such that at most one class is mixed (the repetition stops because the index of the smallest mixed class is strictly increasing after each step). 
If there are $k_1$ classes containing at least one vertex of $G_1$, $k_2$ classes containing at least one vertex of $G_2$, and $k_3$ mixed classes, then 
$k = k_1 + k_2 - k_3$, thus $k_1 + k_2 - k = k_3 \leq 1$. 
Since that ordering and partition restricted to each of $V(G_1)$ and $V(G_2)$ are consistent, $\fp(G_1)+\fp(G_2) \leq k_1+k_2 \leq k+1 = \fp(G_1 \cup G_2)+1$.
\end{proof}

\begin{theorem}[Precedence thinness of join]\label{thm:fp-join}
Let $G_1$ and $G_2$ be graphs. If $G_2$ is complete, then
$\fp(G_1 \vee G_2) = \fp(G_1)$. If neither $G_1$ nor $G_2$ are
complete, then $\fp(G_1 \vee G_2) = \fp(G_1) + \fp(G_2)$.
\end{theorem}

\begin{proof}
The proof is similar to the one of Theorem~\ref{thm:fp-union}. 
Let $\sigma_1$ be an order of $V(G_1)$, and let $V_1, \dots, V_{k_1}$ be a precedence partition of $V(G_1)$, consistent with $\sigma_1$, such that $k_1 = \fp(G_1)$. Define $\sigma_2$ and $W_1, \dots, W_{k_2}$ in the same way for $G_2$. It is not difficult to see that $V_1, \dots, V_{k_1}, W_1, \dots, W_{k_2}$ is a precedence partition of $V(G_1 \vee G_2)$ into $k_1+k_2$ sets, consistent with the order $\sigma_1 \oplus \sigma_2$. So $\fp(G_1 \vee G_2) \leq
\fp(G_1)+\fp(G_2)$. Moreover, if $G_2$ is complete, 
$V_1, \dots, V_{k_1} \cup V(G_2)$ is a precedence partition of $V(G_1 \vee G_2)$ into $k_1$ sets, consistent with the order $\sigma_1 \oplus \sigma_2$.

Now let $\sigma$ be an order of $V(G_1 \vee G_2)$, and let $\mathcal{V} = \{V_1, \dots, V_{k}\}$ be a precedence partition of $V(G_1 \vee G_2)$, consistent with $\sigma$, such that $k = \fp(G_1 \vee G_2)$. We call a class \emph{mixed} if it contains vertices of both $G_1$ and $G_2$, and \emph{uniform} otherwise. 
We will prove that $G_1 \cup G_2$ admits a vertex ordering and a consistent precedence partition into the same number $k$ of classes, such that at most one class is mixed. 
Moreover, if none of $G_1$ and $G_2$ is complete, then no class is mixed. 

Suppose $\mathcal{V}$ contains a mixed class, and let $i$ be the smallest index such that $V_i$ is mixed. Let $V_i^1 = V_i \cap V(G_1)$ and $V_i^2 = V_i \cap V(G_2)$. Suppose that $\first(V_i)$ belongs to $V(G_1)$. Then, $V_i^2$ induces a complete graph and for every $j > i$, the vertices of $V_i^2$ are adjacent to all the vertices of $V_j$: they are adjacent to every vertex in $V(G_1) \cap V_j$ by definition of join, and they cannot have a non-neighbor in $V(G_2) \cap V_j$, since such a vertex should be adjacent to $\first(V_i)$ by definition of join, and that would break consistency. 

If $G_2$ is not a complete graph, then there is another class containing vertices of $G_2$. Let $j$ be the greatest index different from $i$ such that $V_j \cap V(G_2) \neq \emptyset$. If $j > i$, $V_j$ may be either uniform or mixed, while if $j < i$, $V_j$ is uniform. 
Moreover, if $j < i$, the vertices of $V_i^2$ are also adjacent to all the vertices of classes between $V_{j+1}$ and $V_i$ (both included), since those classes contain only vertices of $G_1$. 

We will define a partition $\mathcal{V'}= [\mathcal{V} \setminus \{V_i,V_j\}] \cup \{V_i^1, V_j \cup V_i^2\}$, and the order $\sigma'$ obtained by $\sigma$ by deleting the vertices of $V_i^2$ and reinserting them right after the vertices of $V_j$. 

Let $x <_{\sigma'} y <_{\sigma'} z$ such that $x$ and $y$ are in the same class of $\mathcal{V'}$, $x$ and $z$ are adjacent. Since $\sigma$ and $\mathcal{V}$ are consistent, if $x <_{\sigma} y <_{\sigma} z$ and $x$ and $y$ are in the same class of $\mathcal{V}$, then $z$ is adjacent to $y$. So, suppose first that $x$ and $y$ are not in the same class of $\mathcal{V}$. Then $x \in V_j$ and $y \in V_i^2$. This implies that either $z \in V_i^2$ or $z \in V_\ell$ with $\ell > j$. In either case, $y$ is adjacent to $z$. Suppose now that some of the two inequalities do not hold for $\sigma$. By definition of $\sigma'$, this implies that at least one of the vertices belongs to $V_i^2$. If $y \in V_i^2$, then $y$ is adjacent to $z$, since every vertex of $V_i^2$ is adjacent to every vertex greater than itself, both according to $\sigma$ and according to $\sigma'$. 
If $x \in V_i^2$, since $x$ and $y$ are in the same class of $\mathcal{V'}$, and by the definition of the order and partition, it must be $y \in V_i^2$, thus adjacent to $z$. 
By last, if $z \in V_i^2$, there are two cases. If $z <_{\sigma} y$, then $z$ is adjacent to $y$. If $y <_{\sigma} z$, since we are assuming that some of the inequalities do not hold in $\sigma$, necessarily $y <_{\sigma} x$; and since $x$ and $y$ belong to the same class of $\mathcal{V}'$, it must be that $x \in V_j$ and $y \in V_i^2$. But then, it is clear that $y$ is adjacent to $z$.

The case in which $\first(V_i)$ belongs to $V(G_2)$ is symmetric, and we can repeat this procedure until obtaining an ordering and a consistent partition into $k$ classes, such that at most one class is mixed, and this happens only in the case in which one of the graphs is complete (the repetition of the procedure stops because the index of the smallest mixed class is strictly increasing after each step). 

Suppose that at the end of the process there are $k_1$ classes containing at least one vertex of $G_1$ and $k_2$ classes containing at least one vertex of $G_2$.
In the case in which we have a mixed class, 
$k = k_1 + k_2 - 1$, thus $k_1 + k_2 = k+1$. 
Since that ordering and partition restricted to each of $V(G_1)$ and $V(G_2)$ are consistent, $\fp(G_1)+\fp(G_2) \leq k_1+k_2 = k+1 = \fp(G_1 \cup G_2)+1$.
If none of the graphs is complete, we have no mixed class at the end of the process, so  
$k = k_1 + k_2$. 
Since that ordering and partition restricted to each of $V(G_1)$ and $V(G_2)$ are consistent, $\fp(G_1)+\fp(G_2) \leq k_1+k_2 = k = \fp(G_1 \cup G_2)$.
\end{proof}

\begin{corollary}
\label{thm:fp-cographs}
The precedence thinness of cographs can be computed in polynomial time. 
\end{corollary}

The thinness of the complement of an interval graph can be arbitrarily large, $\overline{tK_2}$ being an example of it, but it can be also equal to one, as for example in the case of complete graphs or edgeless graphs. So, we cannot express the thinness of a graph in terms of the thinness of its complement. However, one can aim to relate the complete thinness of a graph in terms of the independent thinness of its complement, and viceversa. As we will show next, this is not possible for the regular complete/independent thinness or proper thinness (Proposition~\ref{no-rel}), but a relation exists for the precedence versions of complete and independent thinness. 


\begin{proposition}\label{no-rel}
For a matching $nK_2$, $n\geq 2$, it holds that $\indpthin(nK_2) = 2$, while $\compthin(\overline{nK_2}) = n$. For the crown $CR_n$, $n \geq 2$, it holds that $\comppthin(\overline{CR_n}) = 2$, while $\indthin(CR_n) = n$.     
\end{proposition}

\begin{proof} Let $G=nK_2$, $n \geq 2$, and $V(G)=V \cup W$, where $V=\{v_1, \dots, v_n\}$, $W=\{w_1,\dots, w_n\}$, such that the only edges of $G$ are $(v_i,w_i)$ for $1 \leq i \leq n$. Then it is easy to verify that the partition $\{V,W\}$ and the ordering $v_1,w_1,v_2,w_2, \dots, v_n,w_n$ form a strongly consistent solution, and both $V$ and $W$ are independent sets. On the other hand, $\indpthin(nK_2) \geq 2$, since $V(G)$ is not an independent set. As for $\overline{G}$, it is known that $\thin(\overline{nK_2}) = n$ (e.g.~\cite{C-M-O-thinness-man,BGOSS-thin-oper}), and it is not hard to verify that a consistent solution where the classes are complete sets is given by the vertex order $v_1 v_2 w_1 [v_i w_{i-1} v_{i+1} w_i]_{2 < i \leq n-1} w_n$ and the partition $\{\{v_1,w_n\}\} \cup \bigcup_{2 \leq i \leq n}\{\{v_i,w_{i-1}\}\}$. Indeed, the solution is also strongly consistent.  

The equality $\indthin(CR_n) = n$ is proved in Theorem~\ref{thm:crn:indthin}. Now let $G=\overline{CR_n}$, $n \geq 2$, let $V(G)=V \cup W$, where $V=\{v_1, \dots, v_n\}$, $W=\{w_1,\dots, w_n\}$, such that $V$ and $W$ are complete sets and the only edges of $G$ joining vertices of $V$ and $W$ are $(v_i,w_i)$ for $1 \leq i \leq n$. Then it is easy to verify that the partition $\{V,W\}$ and the ordering $v_1,w_1,v_2,w_2, \dots, v_n,w_n$ form a strongly consistent solution, and both $V$ and $W$ are complete sets. On the other hand, $\comppthin(\overline{CR_n}) \geq 2$, since $V(G)$ is not a complete graph.
\end{proof}

\begin{theorem}\label{thm:fp-complement}
Let $G$ be a graph. Then $\indfp(G) = \compfp(\overline{G})$, and 
$\indfpp(G) = \compfpp(\overline{G})$.
\end{theorem}

\begin{proof}
Let $G$ be a graph and $\sigma$ be an order of $V(G)$. Let $V_1, \dots, V_k$ be a precedence partition of 
$V(G)$ into independent sets which is consistent (respectively strongly consistent) with $\sigma$ (i.e., $v <_{\sigma} w$ for $v \in V_i$ and $w \in V_j$ with $1 \leq i < j \leq k$). 

Consider now the same ordered partition $V_1, \dots, V_k$ in $\overline{G}$ (now into complete sets), but where we order internally the vertices of each part according to $\overline \sigma$. We will prove that the partition and the order are consistent (respectively strongly consistent) for $\overline{G}$, implying that $\indfp(G) \geq \compfp(\overline{G})$, and \linebreak
$\indfpp(G) \geq \compfpp(\overline{G})$. 

The (strong) consistency cannot be broken by a triple of vertices of the same class, since each class is a complete set. So let $v < w < z$ according to the order defined for $\overline{G}$, such that $v$ and $w$ are in the same class $V_i$ of the partition, $z$ belongs to a different class, and $(v,z) \in E(\overline{G})$. By the definition of the order, $z \in V_j$ with $j > i$, and $w <_{\sigma} v <_{\sigma} z$. If $(w,z) \not \in E(\overline{G})$, then $(w,z) \in E(G)$, and by consistency of $\sigma$ and $V_1, \dots, V_k$ in $G$, $(v,z) \in E(G)$, a contradiction. Thus, $(w,z) \in E(\overline{G})$. Similarly, if $\sigma$ and $V_1, \dots, V_k$  were strongly consistent, let $v < w < z$ according to the order defined for $\overline{G}$, such that $w$ and $z$ are in the same class $V_j$ of the partition, $v$ belongs to a different class, and $(v,z) \in E(\overline{G})$. By the definition of the order, $v \in V_i$ with $i < j$, and $v <_{\sigma} z <_{\sigma} w$. If $(v,w) \not \in E(\overline{G})$, then $(v,w) \in E(G)$, and by strong consistency of $\sigma$ and $V_1, \dots, V_k$ in $G$, $(v,z) \in E(G)$, a contradiction. Thus, $(v,w) \in E(\overline{G})$.

To prove the converse inequalities, it is enough to observe that the same argument, applied to $\overline{G}$ and exchanging independent sets by complete sets, proves that $\indfp(G) \leq \compfp(\overline{G})$, and 
$\indfpp(G) \leq \compfpp(\overline{G})$.  
\end{proof}

\section{Bounds for the thinness of grids} \label{sec:grid}

In this section, we show that the linear decomposition  in~\cite{VatshelleThesis} for mim-width (a width parameter with several algorithmic applications), leads to a consistent ordering and partition of $V(GR_n)$ into $\left \lceil \frac{n+1}{2} \right \rceil$ classes, thus improving the previously known upper
bound for the thinness of the grid.


\begin{theorem}[Bound for the thinness of $GR_{n,m}$]
    \label{thinness-GRmn}
    For $1 \leq n \leq m$, $\left \lceil \frac{n-1}{3} \right \rceil \leq \thin(GR_{n,m}) \leq  \left \lceil \frac{n+1}{2} \right \rceil$.
\end{theorem}

\begin{proof}
To prove the upper bound, we will define a vertex partition and consistent ordering as follows. 
The first class $V_1$ consists of the union of vertices $(1,j)$, $1 \leq j \leq m$ and the vertices $(2,j)$, $1 \leq j \leq m$, $j$ odd. For $2 \leq i \leq \frac{n}{2}$, the $i$-th class $V_i$ consists of the union of vertices $(2i-1,j)$, $1 \leq j \leq m$ and the vertices $(2i-2,j)$ and $(2i,j)$, $1 \leq j \leq m$, $j \equiv i \mod 2$. The last class $V_i$, $i = \left \lceil \frac{n+1}{2} \right \rceil$, is composed by 
the union of the vertices $(2i-2,j)$, $1 \leq j \leq m$, $j \equiv i \mod 2$, and the vertices $(n,j)$, $1 \leq j \leq m$, in the case in which $n$ is odd. Examples for $n=7$ and $n=6$ can be seen in Figure~\ref{fig:grid}, where the bold edges are the ones that are internal to a class. 

Each class induces an interval graph. We will define first an internal ordering for each class, that satisfies consistency (i.e., a canonical ordering). For the first class, $(x,j) < (x',j')$ if $j < j'$, and $(2,j) < (1,j)$ for $1 \leq j \leq m$, $j$ odd. For the last class, $(x,j) < (x',j')$ if $j < j'$, and $(n-1,j) < (n,j)$ for the corresponding values of $j$, when $n$ is odd. For the $i$-th class, $2 \leq i \leq \frac{n}{2}$, $(x,j) < (x',j')$ if $j < j'$, and $(2i,j) < (2i-2,j) < (2i-1,j)$ for $2 \leq i \leq \frac{n}{2}$, $1 \leq j \leq m$, $j \equiv i \mod 2$. 

There are edges joining $V_i$ and $V_j$ if and only if $|i-j|=1$. So, it is enough to prove that we can combine the defined orderings for two consecutive classes into a single one, consistent with the graph induced by the union of the two classes. In this way, we can insert class by class the vertices, obtaining a total ordering of the vertices of the grid which is consistent with the defined partition. 

Assuming that vertex $(2i,j)$ belongs to $V_{i+1}$, the joint ordering for $V_i \cup V_{i+1}$ follows the pattern 
$(2i+2,j)$, $(2i,j)$, $(2i-2,j-1)$, $(2i-1,j-1)$, $(2i-1,j)$, $(2i,j+1)$, $(2i+1,j)$, $(2i+1,j+1)$, and repeating from $(2i+2,j+2)$ (restricted to the vertices that actually do exist). 
This ordering is illustrated bottom-up in the right part of Figure~\ref{fig:grid}, where the bold edges are the ones that are internal to a class, the horizontal dotted lines are drawn to help in visualizing the vertex ordering, and the vertices in the described pattern are shaded.   

Let us argue consistency (that can be verified also in the figure). Vertices in column $2i-2$ do not have neighbors in $V_{i+1}$, and vertices in column $2i+2$ do not have neighbors in $V_i$. Recall that we are assuming that vertex $(2i,j)$ belongs to $V_{i+1}$. 
So the vertex $(2i-1,j-1)$ has no neighbors in $V_{i+1}$, and the vertex $(2i+1,j)$ has no neighbors in $V_i$. 

The only neighbor of   
$(2i-1,j)$ in $V_{i+1}$ is $(2i,j)$, and the only vertices between those two in the ordering are $(2i-2,j-1)$ and $(2i-1,j-1)$, both belonging to $V_i$. 
The only neighbor of   
$(2i+1,j+1)$ in $V_i$ is $(2i,j+1)$, and the only vertex in between in the ordering is $(2i+1,j)$, belonging to $V_{i+1}$. 

The only neighbor of   
$(2i,j+1)$ in $V_{i+1}$ smaller than itself is $(2i,j)$, and the only vertices between those two in the ordering are $(2i-2,j-1)$, $(2i-1,j-1)$, and $(2i-1,j)$, all of them belonging to $V_i$. 
Finally, the only neighbor of   
$(2i,j)$ in $V_{i}$ smaller than itself is $(2i,j-1)$, and the only vertices between those two in the ordering are $(2i+1,j-2)$, $(2i+1,j-1)$, and $(2i+2,j)$, all of them belonging to $V_{i+1}$. 

To prove the lower bound, we will use a result by Jel\'{\i}nek~\cite{Jelinek-grids}, namely, for every order $v_1, \dots, v_{nm}$ of the vertices of $GR_{n,m}$, there exists an index $t$ such that the induced bipartite graph $G[A,A']$ where $A=\{v_1, \dots, v_t\}$ and $A'=\{v_{t+1}, \dots, v_{nm}\}$, contains a matching of size $n-1$. Let $B \subseteq A$ and $B' \subseteq A'$ the matched vertices, $|B|=|B'|=n-1$. Suppose that there is a partition of $V(GR_{n,m})$ into $k$ classes, consistent with that ordering, suppose that there is a class $W$ such that $|W \cap B| \geq 4$, and let $w < x < y < z$ be four vertices of $W \cap B$. Let $w', x', y', z'$ be their respective matches in $B'$. Notice that all of them are greater than $z$. Then, by consistency, $w'$ is adjacent also to $x$, $y$, and $z$; $x'$ is adjacent also to $y$, and $z$.
But then there are two cycles $C_4$ consisting of $x',y,w',z$ and $x',y,w',x$, having two common edges ($(y,x')$ and $(y,w')$), which is not possible in $GR_{n,m}$. Therefore, $|W \cap B| \leq 3$ for each class $W$, hence $k \geq \frac{n-1}{3}$. 
\end{proof}

\begin{figure}
    \begin{center}
    \begin{tikzpicture}[scale=0.69]

\foreach \x in {1,...,7}
    \foreach \y in {1,...,9}
        \vertex{\x}{\y}{v\x\y};

\foreach \x in {1,...,6}
    \foreach \y in {1,...,9}
        \vertex{8+\x}{\y}{w\x\y};

\foreach \x in {1,...,7}
   \path (v\x1) edge [line width=0.1mm] (v\x9);

   \foreach \y in {1,...,9}
   \path (v1\y) edge [line width=0.1mm] (v7\y);

\foreach \x in {1,...,6}
   \path (w\x1) edge [line width=0.1mm] (w\x9);

   \foreach \y in {1,...,9}
   \path (w1\y) edge [line width=0.1mm] (w6\y);

\foreach \x in {1,3,5,7}
   \path (v\x1) edge [line width=0.5mm] (v\x9); 

\foreach \x in {1,3,5}
   \path (w\x1) edge [line width=0.5mm] (w\x9);    

\foreach \x / \z in {1/2,4/6}
    \foreach \y in {1,3,5,7,9}
           \path (v\x\y) edge [line width=0.5mm] (v\z\y);    

\foreach \x / \z in {2/4,6/7}
    \foreach \y in {2,4,6,8}
           \path (v\x\y) edge [line width=0.5mm] (v\z\y);  

\foreach \x / \z in {1/2,4/6}
    \foreach \y in {1,3,5,7,9}
           \path (w\x\y) edge [line width=0.5mm] (w\z\y);    

\foreach \x / \z in {2/4}
    \foreach \y in {2,4,6,8}
           \path (w\x\y) edge [line width=0.5mm] (w\z\y);  

   \foreach \y in {2,4,6,8}
           \draw (13.75,\y-0.25) -- (14.25,\y-0.25) -- (14.25,\y+0.25) -- (13.75,\y+0.25) -- (13.75,\y-0.25);

    \node at (4,0) {Odd $n$};
    \node at (11.5,0) {Even $n$};

\fill [color=lightgray] (15.5,3.3) -- (19.5,3.3) -- (19.5,4.95) -- (15.5,4.95) -- (15.5,3.3);

   \foreach \y in {1,3,5,7,9}
{      \vertex{16}{1.2+0.85*\y}{z1\y};        
        \draw [dotted] (15.5,1.2+0.85*\y) -- (19.5,1.2+0.85*\y); 
}

    \foreach \y in {1,3,5,7,9}
{      \vertex{17}{0.2+0.85*\y}{z3\y}; 
        \draw [dotted] (15.5,0.2+0.85*\y) -- (19.5,0.2+0.85*\y); 
}

    \foreach \y in {2,4,6,8}
        \vertex{18}{0.2+0.85*\y}{z3\y};

   \foreach \y in {1,3,5,7,9}
      \vertex{16.5}{1.4+0.85*\y}{z2\y};    

   \foreach \y in {2,4,6,8}
      \vertex{16.5}{0.8+0.85*\y}{z2\y};    

   \foreach \y in {1,3,5,7,9}
      \vertex{18.5}{0.6+0.85*\y}{z4\y};    

   \foreach \y in {2,4,6,8}
      \vertex{18.5}{1.2+0.85*\y}{z4\y};    

       \foreach \y in {2,4,6,8}
        \vertex{19}{0.85*\y}{z5\y};

\draw [dotted] (17.5,1) -- (17.5,9); 

   \foreach \x in {2,4}
   \path (z\x1) edge [line width=0.5mm] (z\x9);

    \foreach \y in {1,3,5,7,9}
{   \path (z1\y) edge [line width=0.5mm] (z2\y);
   \path (z3\y) edge [line width=0.5mm] (z2\y);
      \path (z3\y) edge [line width=0.1mm] (z4\y);
}

       \foreach \y in {2,4,6,8}
{   \path (z3\y) edge [line width=0.5mm] (z4\y);
   \path (z5\y) edge [line width=0.5mm] (z4\y);
         \path (z3\y) edge [line width=0.1mm] (z2\y);
}

         \path (z31) edge [line width=0.1mm] (z32);
        \path (z33) edge [line width=0.1mm] (z32);
        \path (z33) edge [line width=0.1mm] (z34);
        \path (z35) edge [line width=0.1mm] (z34);
        \path (z35) edge [line width=0.1mm] (z36);
        \path (z37) edge [line width=0.1mm] (z36);
        \path (z37) edge [line width=0.1mm] (z38);
        \path (z39) edge [line width=0.1mm] (z38);

    \end{tikzpicture}
    \end{center}
    \caption{Scheme for the consistent partition and ordering of a grid in Theorem~\ref{thinness-GRmn}. The third drawing shows the combined ordering of two consecutive classes (bottom up).}\label{fig:grid}
\end{figure}
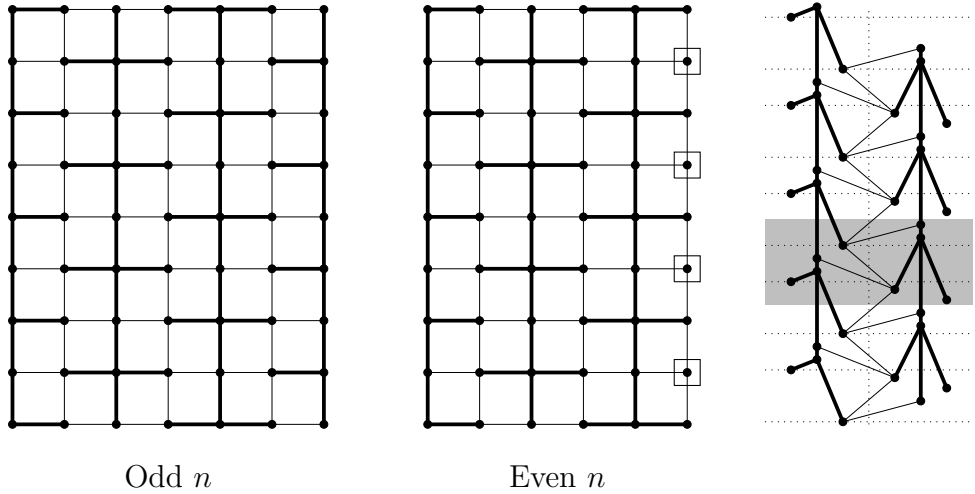

\begin{corollary}
[Bound for the thinness of $GR_n$]
    \label{thinness-GRr}
    For $n \geq 1$, $\left \lceil \frac{n-1}{3} \right \rceil \leq \thin(GR_n) \leq  \left \lceil \frac{n+1}{2} \right \rceil$.    
\end{corollary}


The grid $GR_{2,n}$ can be also viewed as the Cartesian product of the path $P_n$ and the complete graph $K_2$. So, as a corollary of the results in~\cite{B-D-thinness} for thinness of Cartesian products of graphs, $\thin(GR_{2,n})=2$, for $n\geq 2$. 

The precedence thinness instead, grows linearly in $n$ for $GR_{n,2}$ and quadratically in $n$ for $GR_n$, as it can be seen in the following theorems, whose proofs are based on Theorem~\ref{thm:fp-union}.  

\begin{theorem} \label{precthinness-GRn2}
For $n \geq 1$, $\fp(GR_{2,n}) = \left \lceil \frac{n+1}{2} \right \rceil$.    
\end{theorem}

\begin{proof}
Let $GR_{2,n} = (V,E)$ with $V = \{(i,j) : 1 \leq i \leq 2, 1 \leq j \leq n\}$. Let $(\mathcal{V},\sigma)$ be any solution for the precedence thinness of $GR_{2,n}$, with $\mathcal{V} = (V_1, \ldots, V_k)$. 

Let us show first that $|V_1| \leq 2$. Suppose that $|V_1| \geq 3$ and let $v_1,v_2,v_3$ be the three first elements of $V_1$. (Ordering relations are assumed to be according to $\sigma$.) Note that $d(v_1) \geq \delta(GR_{2,n}) = 2$. Let $u_1 < u_2$ be two first neighbors of $v_1$. If $u_1 \in V_1$, then $(u_1,u_2) \in E$ or otherwise $(v_1, u_1, u_2)$ would break the consistency. But then $v_1,u_1,u_2$ induce a $C_3$, which is not possible. Therefore, $u_1 \notin V_1$ and, consequently, all neighbors of $v_1$ are not in $V_1$. Thus, $v_2,v_3$ are also neighbors of $u_1,u_2$, to avoid having $(v_1,v_i,u_j)$ break the consistency, for all $i \in \{2,3\}, j \in \{1,2\}$. As $(u_1,u_2) \notin E$ (avoiding the cycle $v_1,u_1,u_2$), there are two cycles isomorphic to $C_4$ consisting of $v_1,u_1,v_2,u_2$ and $v_1,u_1,v_3,u_2$, having two common edges ($(v_1,u_1)$ and $(v_1,u_2)$), which is not possible in $GR_{2,n}$. Therefore, $|V_1| \leq 2$ and, by previous arguments, $V_1$ must be an independent set. 

We prove that $\fp(GR_{2,n}) = \left \lceil \frac{n+1}{2} \right \rceil$ holds by induction on $n$. Since $\fp(GR_{2,1}) = 1$ and $\fp(GR_{2,2}) = 2$, the result holds for $n \leq 2$. Suppose $n \geq 3$ and that the result holds for all $GR_{2,n'}$ with $n' < n$. Let $(\mathcal{V},\sigma)$ be a solution for the precedence thinness of $GR_{2,n}$, with $\mathcal{V} = (V_1, \ldots, V_k)$. We proceed the proof by analyzing the cases in which $|V_1| = 1$ or $|V_1| = 2$.

If $|V_1| = 2$, then one of the elements of $V_1$ is a  ``corner'' of the grid (a degree-two vertex), and the other is that one that  dominates its neighborhood; without loss of generality, $V_1 = \{(1,1), (2,2)\}$. To see that, let $u < v$ be the two vertices of $V_1$ and let $H = GR_{2,n} \setminus N[v]$. We claim that $u$ must be disconnected from the other vertices of $H$. Indeed, suppose there is $(u,w) \in E(H)$. Then, $w \notin V_1$. To avoid having $(u,v,w)$ break the consistency, we have that $(v,w) \in E$. But then $w$ could not be in $H$, since $w \in N[v]$. Thus, the claim holds. Therefore, $u$ must be a corner of the grid (say, vertex $(1,1)$) and $v$ must the vertex that dominates its neighborhood (say, vertex $(2,2)$).

Let $R = \{(i,j) : 1 \leq i \leq 2, 1 \leq j \leq 2\}$. Let $G = GR_{2,n}$ and $G' = G[V \setminus R]$; therefore, $G'$ is isomorphic to $GR_{2,n-2}$. By induction hypothesis, we have that $\fp(G') = \left \lceil \frac{n-2+1}{2} \right \rceil = \left \lceil \frac{n-1}{2} \right \rceil$. We have that, without loss of generality, $V_1 \subseteq R$. Consequently, $(V_2 \setminus R,\ldots,V_k \setminus R)$ is a partition of $V(G')$ and, considering the ordering $\sigma$ restricted to the elements of $V(G')$, such a partition consists of a solution for the precedence thinness of $G'$. Therefore, $k - 1 \geq \fp(G') = \left \lceil \frac{n-1}{2} \right \rceil$, and therefore $k \geq \left \lceil \frac{n-1}{2} \right \rceil + 1 = \left \lceil  \frac{n-1}{2}+1 \right \rceil = \left \lceil  \frac{n+1}{2} \right \rceil$. Thus, $\fp(GR_{2,n}) \geq \left \lceil  \frac{n+1}{2} \right \rceil$.

If $|V_1|=1$, let $v = (a,b)$ be the unique vertex of $V_1$. If $b=1$ or $b=n$ (that is, $v$ is a corner of the grid), without loss of generality, we may assume $b=1$. Letting $R = \{(i,j) : 1 \leq i \leq 2, 1 \leq j \leq 2\}$, $G = GR_{2,n}$ and $G' = G[V \setminus R]$, we may proceed the same reasoning as in the previous paragraph, to conclude that $\fp(GR_{2,n}) \geq \left \lceil  \frac{n+1}{2} \right \rceil$. On the other hand, if $1 < b < n$, the proof proceeds as follows. Let $V' = \{(i,j) : 1 \leq i \leq 2 \text{, } 1 \leq j < b\}$, $V'' = \{(i,j) : 1 \leq i \leq 2 \text{, } b < j \leq n\}$, $G' = G[V']$, and $G'' = G[V'']$. Therefore, $G'$ is isomorphic to $GR_{2,b-1}$, whereas $G''$ is isomorphic to $GR_{2,n-b}$. By Theorem~\ref{thm:fp-union}, we have that $\fp(G' \cup G'') = \fp(G')+\fp(G'')-1 = \left \lceil \frac{b-1+1}{2} \right \rceil + \left \lceil \frac{n-b+1}{2} \right \rceil - 1$, the latter holding by the induction hypothesis. Therefore, $\fp(G' \cup G'') \geq \left \lceil \frac{b}{2} + \frac{n-b+1}{2} -1 \right \rceil = \left \lceil \frac{n-1}{2} \right \rceil$. Let $R = \{(3-a,b) \}$. As $G' \cup G'' \subset GR_{2,n}$, we have that $(V_2 \setminus R,\ldots,V_k \setminus R)$ is a partition of $V(G' \cup G'')$ and, again considering the respective restriction in $\sigma$, such a partition consists of a solution for the precedence thinness of $G' \cup G''$. Therefore, $k - 1 \geq \fp(G' \cup G'') \geq \left \lceil \frac{n-1}{2} \right \rceil$, and we have $\fp(GR_{2,n}) \geq \left \lceil  \frac{n+1}{2} \right \rceil$.

To show an upper bound on $\fp(GR_{2,n})$, we present the partition of Figure~\ref{fig:precgrid}, for odd and even values of $n$. The respective consistent orderings consists of all the elements that are alone in one part (the elements depicted in a square) coming first, in any order, followed by the vertices of the bold path. Those vertices of the path must be ordered according to any canonical ordering of the induced path. It is straightforward to check that this partition and ordering indeed is a solution for the precedence thinness of $GR_{2,n}$. Regarding the size of the partition, note that for each two  units in $n$, there is one part having only one element. Therefore, there are $\lfloor \frac{n}{2} \rfloor$ such  parts. Considering the part having all the remaining vertices, we have that the size of the partition is $\lfloor \frac{n}{2} \rfloor + 1 = \lfloor \frac{n+2}{2} \rfloor = \left \lceil \frac{n+1}{2} \right \rceil$. Therefore,  $\fp(GR_{2,n}) \leq \left \lceil  \frac{n+1}{2} \right \rceil$.  

\end{proof}

\begin{figure}[htb]
    \begin{center}
    \begin{tikzpicture}[scale=0.69]

\foreach \x in {1,...,7}
    \foreach \y in {1,...,2}
        \vertex{\x}{\y}{v\x\y};

\foreach \x in {1,...,6}
    \foreach \y in {1,...,2}
        \vertex{8+\x}{\y}{w\x\y};

\foreach \x in {1,...,7}
   \path (v\x1) edge [line width=0.1mm] (v\x2);

\foreach \y in {1,...,2}
   \path (v1\y) edge [line width=0.1mm] (v7\y);

\foreach \x in {1,...,6}
   \path (w\x1) edge [line width=0.1mm] (w\x2);

\foreach \y in {1,...,2}
   \path (w1\y) edge [line width=0.1mm] (w6\y);

\path (v12) edge [line width=0.5mm] (v32); 
\path (v31) edge [line width=0.5mm] (v51); 
\path (v52) edge [line width=0.5mm] (v72); 
\foreach \x in {1,3,5,7}
    \path (v\x1) edge [line width=0.5mm] (v\x2); 

\path (w12) edge [line width=0.5mm] (w32); 
\path (w31) edge [line width=0.5mm] (w51); 
\path (w52) edge [line width=0.5mm] (w62); 
\foreach \x in {1,3,5}
    \path (w\x1) edge [line width=0.5mm] (w\x2); 

\foreach \x / \y in {2/1, 4/2, 6/1} {
    \draw (\x-0.25,\y-0.25) -- (\x+0.25,\y-0.25) -- (\x+0.25,\y+0.25) -- (\x-0.25,\y+0.25) -- (\x-0.25,\y-0.25);
    \draw (8+\x-0.25,\y-0.25) -- (8+\x+0.25,\y-0.25) -- (8+\x+0.25,\y+0.25) -- (8+\x-0.25,\y+0.25) -- (8+\x-0.25,\y-0.25);
    }

    \node at (4,0) {Odd $n$};
    \node at (11.5,0) {Even $n$};

    \end{tikzpicture}
    \end{center}
    \caption{Scheme for the partition of a grid in Theorem~\ref{precthinness-GRn2}.}\label{fig:precgrid}
\end{figure}
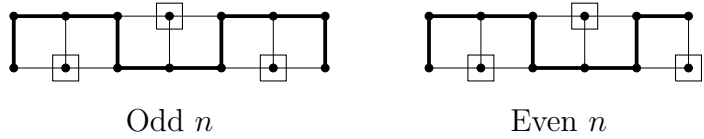

\begin{theorem} \label{precthinness-GRn}
For $n \geq 1$, $\left \lceil \frac{n-1}{3} \right \rceil \left \lceil\frac{n-1}{2}\right \rceil + 1 \leq \fp(GR_n) \leq \left \lceil\frac{n-1}{2}\right \rceil^2 +1$.    
\end{theorem}

\begin{proof}
    Notice that $GR_n$ has the disjoint union of $r = \left \lceil \frac{n-1}{3} \right \rceil$ copies $G_1, \ldots, G_r$ of $GR_{2,n}$ as induced subgraph. By successive applications of Theorem~\ref{thm:fp-union}, we have that $\fp(G_1 \cup \ldots \cup G_r) = \sum_{i=1}^r \fp(G_i) - (r-1) = r \left \lceil \frac{n+1}{2} \right \rceil - r + 1 = r\left \lceil \frac{n-1}{2}\right \rceil +1 = \left \lceil \frac{n-1}{3} \right \rceil \left \lceil\frac{n-1}{2}\right \rceil + 1$. Consequently, $\fp(GR_n) \geq \fp(G_1 \cup \ldots \cup G_r) = \left \lceil \frac{n-1}{3} \right \rceil \left \lceil\frac{n-1}{2}\right \rceil + 1$. 

    For the upper bound, we present the partition depicted in Figure~\ref{fig:precgridn}, consisting of $\left \lceil\frac{n-1}{2}\right \rceil^2 +1$ parts: each set of vertices reached by a same drawing of a ``claw'' is a part, and the caterpillar in the top and right sector of the grid is also a part. (Note that the diagonal edge of the claw is not part of the grid; each claw only represents the vertices belonging to a same part.)  Depending on the position in the grid, the portions of the claws that would cross the left and the bottom sides of the grid are absent, but nevertheless they will be called a claw anyway.
    Also, depending on the parity of $n$, the last part can be either a caterpillar or an independent set, but we will refer to it as a caterpillar. 
    
    We build an ordering $\sigma$ of the vertices of the grid as follows. Enumerate all parts $V_1, \ldots, V_r$ with $r = \left \lceil\frac{n-1}{2}\right \rceil^2$ according to the relative positions among the corresponding claws in the scheme, from bottom to top, and for claws at the same level, from left to right. Therefore, $V_1$ is the bottommost leftmost claw (indeed, the corner and the vertex that dominates its neighborhood), $V_2$ is the first (partial) claw at the right of $V_1$, and so on. For each $1 \leq i \leq r$, if the vertices of $V_i = \{a,b,c,d\}$ are laid out as depicted in the right portion of the figure, then order them in $\sigma$ as $a < b < c < d$. A vertex of a claw in the relative position of $a$ will be called of type $a$, and analogously for types $b,c,d$. Moreover, for all $u \in V_i$, $v \in V_j$, we let $u < v \iff i < j$. Finally, define a new part $V_{r+1}$ having all remaining vertices. The vertices in $V_{r+1}$ are ordered according to a canonical order of the caterpillar. Besides, all vertices in $V_{r+1}$ come in $\sigma$ after all other vertices of the grid. 
    
    It is easy to verify that all vertices belonging to a same part are consecutive in $\sigma$. Let us show that $\sigma$ is consistent to $(V_1,\ldots,V_{r+1})$. Suppose there is $(u,v,w)$ breaking the consistency. It is trivial to check that $u,v,w$ cannot all be in any $V_i$, $1 \leq i \leq r$, and also cannot all be in $V_{r+1}$ since we have used a canonical order. Therefore, $u,v \in V_i$ and $w \in V_j$ with $i \neq j$. Since $v < w$, then $i < j$. 
    
    Notice that, by the chosen ordering among classes, vertices of type $a$ and $b$ of $V_i$ have no neighbors in $V_j$, and all possible neighbors in $V_j$ of the vertex of type $c$ of $V_i$ are also neighbors of the vertex of type $d$. Since $a < b < c < d$ in $V_i$, it is not possible to break consistency with vertices $u,v \in V_i$ and $w \in V_j$, where $i < j$.    
    Thus, $\sigma$ is consistent with the partition. Therefore, $\fp(GR_n) \leq \left \lceil\frac{n-1}{2}\right \rceil^2 +1$.
\end{proof}

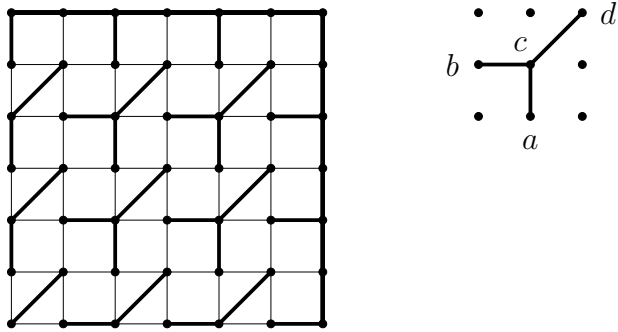
\begin{figure}[htb]
    \begin{center}
    \begin{tikzpicture}[scale=0.69]

\foreach \x in {0,...,6}
    \foreach \y in {0,...,6}
        \vertex{\x}{\y}{v\x\y};

\foreach \x in {9,...,11}
    \foreach \y in {4,...,6}
        \vertex{\x}{\y}{w\x\y};
\path (w105) edge [line width=0.5mm] ($(w105)+(-1,0)$); 
\path (w105) edge [line width=0.5mm] ($(w105)+(0,-1)$); 
\path (w105) edge [line width=0.5mm] ($(w105)+(1,1)$);
\node at (9-0.5,5) {$b$};
\node at (10,4-0.5) {$a$};
\node at (10-0.2,5+0.4) {$c$};
\node at (11+0.5,6) {$d$};

\foreach \x in {0,...,6}
   \path (v\x0) edge [line width=0.1mm] (v\x6);

\foreach \y in {0,...,6}
   \path (v0\y) edge [line width=0.1mm] (v6\y);

\foreach \x in {0,2,4,6}
\foreach \y in {0,2,4,6} {
    \ifthenelse{\x>0}{
   \path (v\x\y) edge [line width=0.5mm] ($(v\x\y)+(-1,0)$); 
    }{};
    \ifthenelse{\y>0}{
    \path (v\x\y) edge [line width=0.5mm] ($(v\x\y)+(0,-1)$); 
    }{};
    \ifthenelse{\x<6 \AND \y<6}{
        \path (v\x\y) edge [line width=0.5mm] ($(v\x\y)+(1,1)$); 
    }{};

    \path (v06) edge [line width=0.5mm] (v66);
    \path (v60) edge [line width=0.5mm] (v66);

}



    \end{tikzpicture}
    \end{center}
    \caption{Scheme for the partition of a grid in Theorem~\ref{precthinness-GRn}.}\label{fig:precgridn}
\end{figure}

\section{Coloring of thin graphs} \label{sec:color}

In this section, we prove that  $k$-coloring ($k$ part of the input) is NP-complete for precedence 2-thin graphs and for proper 2-thin graphs, while it is polynomial-time solvable for precedence proper 2-thin graphs, given the order and partition. 

In order to prove the hardness results, we will reduce from $\mu$-coloring, which is NP-complete on proper interval graphs~\cite{B-F-O-mu-col-full}. 

Given a graph $G$ and a function $\mu:V(G)\to\N$, the
\emph{$\mu$-coloring} problem consists in deciding whether $G$ is
\emph{$\mu$-colorable}, i.e. whether there exists a coloring $f:V(G)\rightarrow \N$ such that $f(v) \leq \mu(v)$ for
every $v\in V(G)$.  

\begin{theorem}
The $k$-coloring problem is NP-complete for precedence 2-thin graphs and for proper 2-thin graphs, when $k$ part of the input.
\end{theorem}

\begin{proof}
Let $(G,\mu')$ be a $\mu'$-coloring instance, where $G$ is a proper interval graph and $v_1, \dots, v_n$ is a proper interval order of $V(G)$. Notice that the instance is equivalent, in terms of feasibility, to the instance $(G,\mu)$ with $\mu(v) = \min(\mu'(v),n)$ for every $v \in V(G)$. We will firstly show a polynomial time reduction from $(G,\mu)$ to an $n$-coloring instance of a precedence 2-thin graph $G'$, and secondly show a polynomial time reduction from $(G,\mu)$ to an $n$-coloring instance of a proper 2-thin graph $G''$.
          
Let $G'$ be a graph such that $V(G') = V(G) \cup A$, where $V(G)$ induces $G$, $A = \{w_1, \dots, w_n\}$ induces a complete graph, and $v \in V(G)$ is adjacent to $w_i$ if and only if $\mu(v) < i$. Let us see that $G'$ is precedence 2-thin with partition $\{A, V(G)\}$, and the vertex order $w_1, \dots, w_n, v_1, \dots, v_n$.  Let $x < y < z$ in $V(G')$ such that $x, y$ belong to the same class and $(x,z) \in E(G')$. If $x, y \in V(G)$, then $(y,z) \in E(G) \subseteq E(G')$ because the order restricted to $V(G)$ is an interval order. If $x, y \in A$, then $x = w_i$ and $y = w_j$ with $i < j$. If $z \in A$, then $(y,z) \in E(G')$ because $A$ induces a complete graph in $G'$. If $z$ in $V(G)$, since $(x,z) \in E(G')$, $\mu(z) < i < j$, so $(y,z) \in E(G')$ as well. 

Now suppose there is a $\mu$-coloring of $G$. We can extend it to an $n$-coloring of $G'$ by giving color $i$ to $w_i$, for $i=1, \dots, n$. Since $v \in V(G)$ is adjacent to $w_i$ if and only if $\mu(v) < i$, the coloring is valid. 
Conversely, suppose there is an $n$-coloring of $G'$. Since $A$ is a complete subgraph of $G'$, we can rename the colors and obtain a coloring $f$ such that $f(w_i) = i$ for $i=1, \dots, n$. Since $v \in V(G)$ is adjacent to $w_i$ if and only if $\mu(v) < i$, for every vertex $v$ of $V(G)$, $f(v) \leq \mu(v)$, so $G$ admits a $\mu$-coloring.

Let $G''$ be a graph such that $V(G'') = V(G) \cup B$, where $V(G)$ induces $G$, $B = w^1_1, \dots, w^1_n, w^2_1, \dots, w^2_n, w^n_1, \dots, w^n_n$ is a  proper interval graph whose  maximal cliques are $w^i_1, \dots, w^i_n$, for $i=1,\dots,n$, and $w^i_j, \dots, w^{i+1}_{j-1}$ for $i=1,\dots,n-1$, $j=2,\dots,n$, and $v_k \in V(G)$ is adjacent to $w^i_j$ if and only if $i=k$ and $\mu(v_k) < j$. Let us see that $G''$ is proper 2-thin with partition $V(G),B$, and the vertex order $w^1_1, \dots, w^1_n, v_1, w^2_1, \dots, w^2_n, v_2, \dots, w^n_1, \dots, w^n_n, v_n$. Let $x < y < z$ in $V(G'')$ such that $x, y$ belong to the same class and $(x,z) \in E(G'')$. If the three vertices belong to the same class, $(y,z) \in E(G'')$ because the order restricted to each class is a proper interval order. Otherwise, $z$ in $V(G)$ and $x,y \in B$, since no vertex of $B$ has a neighbor smaller than itself in $V(G)$. Let $x=w^i_j$ and $y=w^{i'}_{j'}$ and $z=v_k$. Since  $(x,z) \in E(G'')$, $i=k$ and $\mu(z) < j$, so, by the order definition, 
$i'=k$ and $j < j'$, thus $\mu(z) < j$ and $(y,z) \in E(G'')$ as well. It is easy to see that the case in which $x < y < z$ in $V(G'')$, $y, z$ belong to the same class, and $(x,z) \in E(G'')$, may only arise when the three vertices belong to the same class, so $(x,y) \in E(G'')$ because the order restricted to each class is a proper interval order.

Now suppose there is a $\mu$-coloring of $G$. We can extend it to an $n$-coloring of $G''$ by giving color $j$ to $w^i_j$, for $i,j \in \{1, \dots, n\}$. Since $v_k \in V(G)$ is adjacent to $w^i_j$ if and only if $i=k$ and $\mu(v) < j$, the coloring is valid. 
Conversely, suppose there is an $n$-coloring of $G''$. By the definition of the maximal cliques of $G''[B]$, the $n$ color classes are $\{w^1_j,w^2_j,\dots,w^n_j\}$ for $j=1,\dots,n$. So, we can rename the colors and obtain a coloring $f$ such that $f(w^i_j) = j$ for $i,j \in \{1, \dots, n\}$. Since $v_k \in V(G)$ is adjacent to $w^i_j$ if and only if $i=k$ and $\mu(v) < j$, for every vertex $v_k$ of $V(G)$, $f(v_k) < \mu(v_k)$, so $G$ admits a $\mu$-coloring.
\end{proof}

\begin{figure}
    \begin{center}
    \begin{tikzpicture}[scale=0.65]

    \vertex{0}{0}{v0};
    \vertex{1}{0}{v1};
    \vertex{2}{0}{v2};

    \vertex{4}{0}{w0};
    \vertex{5}{0}{w1};
    \vertex{6}{0}{w2};

    \vertex{8}{0}{z0};
    \vertex{9}{0}{z1};
    \vertex{10}{0}{z2};
    
\path (v0) edge [bend left=45] (v2); 

\path (w0) edge [bend left=45] (w2); \path (w0) edge  (w1);

\path (z0) edge [bend left=45] (z2); 
\path (z1) edge  (z2);

    \node at (1,-1) {$S_1$};
    \node at (5,-1) {$S_2$};
    \node at (9,-1) {$S_3$};

    \end{tikzpicture}
    \end{center}
    \caption{The forbidden ordered induced subgraphs for an interval order ($\{S_1,S_2\}$) and a proper interval order ($\{S_1,S_2,S_3\}$).}\label{fig:patternsInt}
\end{figure}
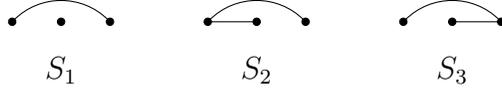

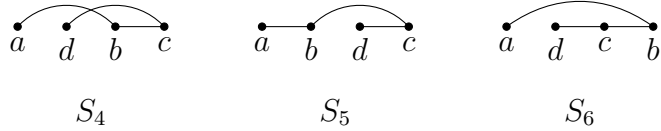
\begin{figure}
    \begin{center}
    \begin{tikzpicture}[scale=0.65]

    \vertex{0}{0}{v0};
    \vertex{1}{0}{v1};
    \vertex{2}{0}{v2};
    \vertex{3}{0}{v3};

    \vertex{5}{0}{z0};
    \vertex{6}{0}{z1};
    \vertex{7}{0}{z2};
    \vertex{8}{0}{z3};

    \vertex{10}{0}{w0};
    \vertex{11}{0}{w1};
    \vertex{12}{0}{w2};
    \vertex{13}{0}{w3};

    \vertexLabel[below]{v0}{$a$};
    \vertexLabel[below]{v1}{$d$};
    \vertexLabel[below]{v2}{$b$};
    \vertexLabel[below]{v3}{$c$};

    \vertexLabel[below]{w0}{$a$};
    \vertexLabel[below]{w1}{$d$};
    \vertexLabel[below]{w2}{$c$};
    \vertexLabel[below]{w3}{$b$};

    \vertexLabel[below]{z0}{$a$};
    \vertexLabel[below]{z1}{$b$};
    \vertexLabel[below]{z2}{$d$};
    \vertexLabel[below]{z3}{$c$};

\path (v0) edge [bend left=45] (v2); 
\path (v1) edge [bend left=45] (v3); 
\path (v2) edge  (v3); 

\path (w0) edge [bend left=35] (w3); 
\path (w2) edge  (w3); 
\path (w1) edge  (w2);

\path (z1) edge [bend left=45] (z3);
\path (z0) edge  (z1);
\path (z2) edge  (z3);

    \node at (1.5,-1.75) {$S_4$};
    \node at (6.5,-1.75) {$S_5$};
    \node at (11.5,-1.75) {$S_6$};

    \end{tikzpicture}
    \end{center}
    \caption{The forbidden ordered induced subgraphs for a perfect order.}\label{fig:patternsPerf}
\end{figure}

A vertex order $<$ of $G$ is \emph{perfect} if $G$ contains no $P_4$ $abcd$ with $a < b$ and $d < c$. A graph is \emph{perfectly orderable} if it admits a perfect order. Perfectly ordered graphs can be optimally colored by applying the greedy coloring algorithm in the perfect order~\cite{Ch-perf-ord}.  

The polynomiality of the $k$-coloring problem on precedence proper 2-thin graphs, given the order and partition, is then a consequence of the following result.

\begin{theorem}
Precedence proper 2-thin graphs are perfectly orderable. Moreover, given a partition of $V(G)$ into two sets $V^1$, $V^2$, strongly consistent with an order $<$ such that every vertex of $V^1$ is smaller than every vertex of $V^2$, a perfect order $\prec$ of $G$ can be obtained by taking the vertices on $V^1$ ordered according to the reverse of $<$, followed by the vertices of $V^2$ ordered according to $<$.
\end{theorem}

\begin{proof}
Suppose $(G,\prec)$ contains one of the ordered induced subgraphs $S_4$, $S_5$ or $S_6$ (Figure~\ref{fig:patternsPerf}). Since $\prec$ restricted to both $V^1$ and $V^2$ is a proper interval order, it does not contain any of the ordered induced subgraphs $S_1$, $S_2$ or $S_3$ (Figure~\ref{fig:patternsInt}). It follows that in any of the cases, the first vertex of the ordered subgraph of Figure~\ref{fig:patternsPerf} belongs to $V^1$ and the last vertex belongs to $V^2$. Moreover, in the case of $S_4$ and $S_5$ the first two vertices belong to $V^1$, and in the case of $S_4$ the last two vertices belong to $V^2$. 

Suppose first that $(G,\prec)$ contains $S_4$. Then $a, d \in V^1$ and $b,c \in V^2$, so $d < a < c$, $(d,c) \in E(G)$ and $(a,c) \not \in E(G)$, a contradiction because $<$ is strongly consistent with the partition $V^1$, $V^2$ in $G$. 

Suppose now that $(G,\prec)$ contains $S_5$. Then $a, b \in V^1$ and $c \in V^2$, so $b < a < c$, $(b,c) \in E(G)$ and $(a,c) \not \in E(G)$, a contradiction because $<$ is strongly consistent with the partition $V^1$, $V^2$ in $G$.

Finally, suppose that $(G,\prec)$ contains $S_6$. If $c \in V^1$, then $d \in V^1$, $c < d < b$, $(c,b) \in E(G)$ and $(d,b) \not \in E(G)$, a contradiction to the strong consistency of $<$ and $V^1$, $V^2$ in $G$. If $c \in V^2$, then $a < c < b$, $(a,b) \in E(G)$ and $(a,c) \not \in E(G)$, a contradiction to the strong consistency of $<$ and $V^1$, $V^2$ in $G$.

Therefore, $\prec$ is a perfect order for $G$.  
\end{proof}

The complexity of the $k$-coloring problem ($k$ part of the input) remains open for precedence proper $t$-thin graphs, $t \geq 3$.

\section{Conclusion} \label{sec:conclusion}

Interval graphs and proper interval graphs are well known graph classes, to which hundreds of papers have been dedicated. Due to its importance, and since not all graphs are (proper) interval graphs, several authors have defined larger classes of graphs by relaxing the definition of (proper) interval graphs or some characterization for this class. This paper aims to one of such generalizations, namely, the (proper) $k$-thin graphs and its variations.

We provide several properties of such generalizations, some of them restricted to the special classes of crown and grid graphs, in order to present the exact values for (proper) thinness, (proper) independent thinness, (proper) complete thinness, precedence (proper) thinness, precedence (proper) independent thinness, and precedence (proper) complete thinness for the crown graphs $CR_n$. The exact values are provided in Table~\ref{table1}.
For cographs, we prove that the precedence thinness can be determined in polynomial time. In particular, we compute the precedence thinness of the union and join of two graphs in terms of their own precedence thinness.  
We also provide bounds for the thinness of $n \times n$ grids $GR_n$ and $n \times m$ grids $GR_{n,m}$, proving that for $1 \leq n \leq m$, $\left \lceil \frac{n-1}{3} \right \rceil \leq
\thin(GR_{n,m}) \leq  \left \lceil \frac{n+1}{2} \right \rceil$. For precedence thinness, it is proved that $\fp(GR_{2,n}) = \left \lceil \frac{n+1}{2} \right \rceil$ and that $\left \lceil \frac{n-1}{3} \right \rceil \left \lceil\frac{n-1}{2}\right \rceil + 1 \leq \fp(GR_n) \leq \left \lceil\frac{n-1}{2}\right \rceil^2 +1$. 

Regarding applications, we show that the $k$-coloring problem ($k$ being part of the input) is NP-complete for both precedence $2$-thin graphs and proper $2$-thin graphs. On the positive side, it is polynomially solvable for precedence proper $2$-thin graphs, given the order and partition. This last result is obtained by showing that precedence proper $2$-thin graphs are perfectly orderable graphs, which are optimally colored by the greedy algorithm on any of its corresponding perfect orderings, and that from the order and partition in the precedence proper $2$-thin representation a perfect order can be computed.


\section*{Acknowledgments}

We thank the two anonymous reviewers whose comments and suggestions have helped us improve this manuscript.
The first author received support from CONICET (PIP 11220200100084CO), 
ANPCyT (PICT-2021-I-A-00755) and UBACyT (20020170100495BA and 20020220300079BA). The third, fourth, and  sixth authors were partially supported by CAPES, CNPq, and FAPERJ. 


\end{document}

%% file: graphs.tikzstyles.tex
\tikzstyle{nodo}=[fill=black, draw=black, shape=circle, inner sep=1.5pt]
\tikzstyle{class 1}=[fill=red, draw=red, shape=circle, inner sep=1.5pt]
\tikzstyle{class 2}=[fill=blue, draw=blue, shape=circle, inner sep=1.5pt]
\tikzstyle{class 3}=[fill=green, draw=green, shape=circle, inner sep=1.5pt]
\tikzstyle{class 4}=[fill=magenta, draw=magenta, shape=circle, inner sep=1.5pt]
\tikzstyle{class 5}=[fill={rgb,255: red,255; green,128; blue,0}, draw={rgb,255: red,255; green,128; blue,0}, shape=circle, inner sep=1.5pt]
\tikzstyle{class 6}=[fill=yellow, draw=yellow, shape=circle, inner sep=1.5pt]

\tikzstyle{arista}=[-, fill=none, draw=black]

%% file: graphs-tikz.tex
\newcommand{\vertex}[4][black]{
    \draw[#1, fill=#1, inner sep=0pt] (#2, #3) circle (0.075) node(#4){};
}

\newcommand{\vertexLabel}[3][above]{
    \path (#2) node[#1]{#3};
}

\newcommand{\edge}[3][]{
    \draw[#1] (#2) -- (#3);
}

\newcommand{\vertexRegularPolygon}[7][]{
    \vertex[#1]
    { { (#2) + (#4) * -cos(deg(6.283*(#6 - 1)/(#5)-(1.571+3.142/(#5)))) } }
    { { (#3) + (#4) * sin(deg(6.283*(#6 - 1)/(#5)-(1.571+3.142/(#5)))) } }
    {#7};
}

\newcommand{\fatSpider}[6][]{
    \foreach \fatSpiderI in {1,...,#5} {
        \vertexRegularPolygon[#1]
                        {#2}{#3}
                        {#4}
                        {#5}
                        {\fatSpiderI}
                        {#6c\fatSpiderI}
                        {};

        \vertexRegularPolygon[#1]
                        {#2}{#3}
                        {2*(#4)}
                        {#5}
                        {\fatSpiderI}
                        {#6s\fatSpiderI}
                        {};
    }

    \foreach \fatSpiderI in {2,...,#5} {
        \pgfmathtruncatemacro{\fatSpiderIminusOne}{\fatSpiderI - 1};

        \foreach \fatSpiderJ in {1,...,\fatSpiderIminusOne} {
            \edge[#1]{#6c\fatSpiderI}{#6c\fatSpiderJ};
        }
    }

    \foreach \fatSpiderI in {1,...,#5} {
        \foreach \fatSpiderJ in {1,...,#5} {
            \ifthenelse
            {\equal{\fatSpiderI}{\fatSpiderJ}}
            {}
            {\edge[#1]{#6s\fatSpiderI}{#6c\fatSpiderJ};};
        }
    }
}